\documentclass[11pt]{amsart}
\usepackage[colorlinks,linkcolor=blue,urlcolor=cyan,citecolor=blue,pagebackref]{hyperref}
\usepackage{amssymb}
\usepackage{amsmath}
\usepackage{mathtools}
\usepackage{mdwtab}
\usepackage{booktabs}
\usepackage[alphabetic,initials]{amsrefs}
\usepackage{caption}
\usepackage{subcaption}
\usepackage{eucal}
\usepackage{bm}
\usepackage{stmaryrd}
\usepackage{tikz}
\usepackage{tikz-cd}
\usetikzlibrary{positioning,fit}
\usepackage[indent]{parskip}
\usepackage{url}
\usepackage{pstricks,pst-node}
\usepackage{fullpage}
\usepackage{float}

\theoremstyle{plain}
\newtheorem{thm}{Theorem}[section]
\newtheorem{prop}[thm]{Proposition}
\newtheorem{lemma}[thm]{Lemma}
\newtheorem{cor}[thm]{Corollary}

\theoremstyle{definition}
\newtheorem{definition}[thm]{Definition}
\newtheorem{example}[thm]{Example}

\newtheorem{claim}[thm]{Claim}

\newtheorem{rem}[thm]{Remark}
\numberwithin{equation}{section}

\DeclareMathAlphabet\mathsf{OT1}{cmss}{m}{n}

 \definecolor{mygray}{rgb}{0.900,0.900,0.900}

\allowdisplaybreaks

\newcommand{\R}{\mathbb R}           
\newcommand{\C}{\mathbb C}           
\newcommand{\Z}{\mathbb Z}           
   
\newcommand{\Ad}{\operatorname{Ad}}

\newcommand{\rank}{\operatorname{rank}}
\newcommand{\spa}{\operatorname{span}}
\newcommand{\Hom}{\operatorname{Hom}}
\newcommand{\Gr}{\underline{\operatorname{Gr}}\,}
\newcommand{\supp}{\operatorname{supp}}

\newcommand{\card}{\operatorname{card}}
\def\AV{\operatorname{AV}}
\def\length{\operatorname{length}}
\newcommand{\width}{\operatorname{width}}
\newcommand{\GKD}{\operatorname{GKdim}}
\newcommand{\sym}{\operatorname{Sym}}

\def\O{\mathcal{O}}

\newcommand{\f}[1]{{\mathfrak{#1}}}           
\newcommand{\fb}{{\mathfrak b}}
\newcommand{\fe}{{\mathfrak e}}
\newcommand{\fg}{{\mathfrak g}}
\newcommand{\fh}{{\mathfrak h}}
\newcommand{\fk}{{\mathfrak k}}
\newcommand{\fl}{{\mathfrak l}}

\newcommand{\fn}{{\mathfrak n}}

\newcommand{\fp}{{\mathfrak p}}
\newcommand{\fq}{{\mathfrak q}}

\newcommand{\fu}{{\mathfrak u}}
\newcommand{\fs}{{\mathfrak s}}

\def\su{{\mathfrak s}{\mathfrak u}}
\def\sp{{\mathfrak s}{\mathfrak p}}
\def\so{{\mathfrak s}{\mathfrak o}}

\newcommand{\ga}{\alpha}
\newcommand{\gb}{\beta}

\newcommand{\gl}{\lambda}
\newcommand{\gL}{\Lambda}
\newcommand{\gd}{\delta}
\newcommand{\gD}{\Delta}

\renewcommand{\gg}{\gamma}

\newcommand{\eps}{\varepsilon}
\newcommand{\ep}{\varepsilon}

\newcommand{\Eul}{\EuScript}

\newcommand{\fB}{{\mathfrak B}}
\newcommand{\fC}{{\Eul C}}

\newcommand{\cN}{{\mathcal N}}
\newcommand{\cO}{{\mathcal O}}
\newcommand{\cU}{\mathcal U}
\newcommand{\cW}{\mathcal W}

\renewcommand{\bar}[1]{\overline{#1}}

\newcommand{\od}{^{\mathrm{odd}}}
\newcommand{\ev}{^{\mathrm{ev}}}

\newcommand{\IP}[2]{\langle#1\,, #2\rangle}     

\begin{document}

\setcounter{page}{0}

\title[Highest weight Harish-Chandra modules]{On the associated variety of a highest weight Harish-Chandra module}

\author{Zhanqiang Bai}
\address{School of Mathematical Sciences, Soochow University, Suzhou 215006, Jiangsu, China}
\email{\tt  zqbai@suda.edu.cn}
  
\author{Markus Hunziker}
\address{Department of Mathematics,
Baylor University, Waco, TX 76798, USA}
\email{\tt Markus\underline{\ }Hunziker@baylor.edu}
  
\author{Xun Xie}  
\address{School of Mathematics and Statistics, Beijing Institute of Technology, Beijing 100081, China}
\email{\tt xieg7@163.com}
  
\author{Roger Zierau}
\address{Mathematics Department, Oklahoma State University,
Stillwater, OK 74078, USA}
\email{\tt roger.zierau@okstate.edu}



\begin{abstract}
We prove a simple formula that calculates the associated variety of a highest weight 
Harish-Chandra module directly from its highest weight. We also give a formula for 
the Gelfand--Kirillov dimension of highest weight Harish-Chandra module which is 
uniform across  Cartan types and is valid for arbitrary infinitesimal character.
\end{abstract}

\setcounter{page}{1} 
 
\maketitle

\tableofcontents

\section{Introduction}

\subsection{Highest weight Harish-Chandra modules}\label{sec:intro-1}  
Let $G_\R$ be a simple Lie group.  It follows from the work of  Harish-Chandra  (\cites{HC55,HC56}) that there exist infinite-dimensional highest weight Harish-Chandra  modules for $G_\R$ if and only if $G_\R$ is of Hermitian type.  Our main result, Theorem \ref{thm:main}, is a very simple formula for the associated varieties of highest weight Harish-Chandra  modules.  Therefore, we assume that $G_\R$ is of Hermitian type. We also assume that $G_\R$ has finite center.

Standard Lie theory tells us that $\fg_\R={\Eul Lie}(G_\R)$ has a Cartan decomposition $\fg_\R=\fk_\R\oplus\fp_\R$ and the subgroup $K_\R$ of $G_\R$ with Lie algebra $\fk_\R$ is compact.  Therefore, $K_\R$ has a complexification $K$.  Writing the complexified Cartan decomposition as $\fg=\fk\oplus\fp$ (since $G_\R$ is of Hermitian type) there is a decomposition $\fp=\fp^+\oplus\fp^-$ into nonzero irreducible $K$-subrepresentations.  This gives a triangular decomposition $\fg=\fp^-\oplus\fk\oplus\fp^+$ and a maximal parabolic subalgebra $\fq=\fk\oplus\fp^+$.  It follows from  Harish-Chandra  that a highest weight Harish-Chandra  module is the irreducible quotient of a Verma module and the highest weight $\gl$ is integral for the roots in $\fk$ (with respect to a Cartan subalgebra $\fh$ of $\fg$ contained in $\fk$).  See \S\ref{ssec:hwhc} for precisely which Verma modules have irreducible quotients occurring as highest weight Harish-Chandra  modules.

\subsection{Associated varieties}
The associated variety of a $(\fg,K)$-module, as defined in \cite{Vogan81}, is a union of (closures of) nilpotent $K$-orbits in $(\fg/\fk)^*$.  When $L(\gl)$ is a highest weight Harish-Chandra  module for $G_\R$, the associated variety is contained in $(\fg/\fq)^*\simeq\fp^+$.  It is well-known that the closures of the $K$-orbits in $\fp^+$ form a chain
\begin{equation}\label{eqn:chain}
\{0\}=\overline{\mathcal{O}_{0}} \subsetneq  \overline{\mathcal{O}_{1}} \subsetneq  \cdots   \subsetneq  \overline{\mathcal{O}_{r}} =\fp^+\!,
\end{equation}
where $r$ is the real rank of $\fg_\R$ ($=\rank(G_\R/K_\R)$).  It follows that 
\begin{equation}\label{eqn:av-intro}
\AV(L(\gl))=\bar{\cO_k}, \text{ for some }k=0,1,\dots,r.
\end{equation}
The purpose of this paper is to express this integer $k=k(\gl)$ in terms of the highest weight $\gl$ in a simple way that is uniform across  Cartan types and is valid for arbitrary infinitesimal character. In the special case when $L(\lambda)$ 
is the Harish-Chandra module of a \emph{unitary} representation of $G_\R$ this was accomplished in 
\cite{BH:15} by using previous work of Joseph \cite{Jos:92}.

\subsection{Statement of the main result} 
To state our main result precisely,  we need some notation for sets of roots.
Let $\Delta$ and $\Delta(\fk)$ denote the root systems of $(\fg,\fh)$ and $(\fk,\fh)$, respectively. 
Let $\Delta^+$ be a positive system of $\Delta$ containing $\gD(\fp^+)$ and define 
$\Delta^+(\fk)=\Delta(\fk)\cap \Delta^+$. Let $\Lambda^+(\fk)$ denote the lattice of integral and dominant weights 
for $\gD(\fk)$.
We  view all sets of roots as partially ordered sets via the usual partial ordering, where 
$\alpha\leq \beta$ means that $\beta - \alpha$ is nonnegative integer linear combination 
of positive roots. 

\begin{definition}
For $\lambda \in \Lambda^+(\fk)$,
define the \emph{diagram} of $\lambda$ as the set
\begin{equation}
Y_{\lambda} :=\{\alpha\in \Delta(\fp^+)\mid \IP{\lambda +\rho}{\alpha^\vee}\in \Z_{\leq 0}\},
\end{equation}
viewed as a subposet of $\Delta_\lambda(\fp^+):=\Delta_\lambda \cap \Delta(\fp^+)$, where 
$\Delta_\lambda:=\{\alpha\in \Delta \mid \IP{\lambda+\rho}{\alpha^\vee}\in \Z\}$
is the integral root system associated to $\lambda$ as usual.
\end{definition}
An \emph{antichain} in a poset is a subset consisting of pairwise noncomparable elements.
The \emph{width} of a poset is the cardinality of maximal antichain in the poset.
Our main result states that the 
associated variety of a highest weight Harish-Chandra module $L(\lambda)$ is completely determined by the width of $Y_\lambda$.

\begin{thm}\label{thm:main}
Suppose $L(\lambda)$  is a highest weight Harish-Chandra module with highest weight $\lambda$ and $\AV(L(\lambda))=\overline{\mathcal{O}_{k(\lambda)}}$. Let 
$m=\operatorname{width}(Y_\lambda)$. Then $k(\gl)$ is given as follows.
\begin{enumerate}
 \item[(a)] If $\gD$ is simply laced and $\gl$ is integral, then $k(\gl)=m$.
 \item[(b)] If $\gD$ is non-simply laced and $\gl$ is integral, then
 \begin{equation*}
  k(\gl)=\begin{cases} 2m, &\text{if }m< \frac{r+1}{2}  \\
                       r, &\text{if }m=\frac{r+1}{2}.
         \end{cases}
 \end{equation*}
\item[(c)] If $\gD$ is non-simply laced and $\gl$ is half-integral, then
 \begin{equation*}
  k(\gl)=\begin{cases} 2m+1, &\text{if }m< \frac{r}{2}  \\
                       r, &\text{if }m=\frac{r}{2}.
         \end{cases}
 \end{equation*}
\item[(d)] In all other cases $k(\gl)=r$.
\end{enumerate}
\end{thm}

At first glance, it may seem  difficult to compute $\operatorname{width}(Y_\lambda)$. However, this is not the case.

\begin{example}
Suppose $G_\R$ is the metaplectic double cover of the symplectic group $Sp(2n,\R)$  and $L(\lambda)$  
is a highest weight Harish-Chandra module with half-integral highest weight $\lambda$.
Using the conventions for roots and weights as in \cite{EHW:83}, 
we then  have  $\lambda+\rho=(t_1,\ldots,t_n)$, 
where the entries $t_i$ are  half-integers such that  $t_1> \cdots >t_n$.
Since $\Delta(\fp^+)=\{\eps_i +\eps_j \mid 1\leq i < j\leq n\}\cup \{2\eps_i \mid 1\leq i\leq n\}$ and two short roots $\eps_{i_1}+\eps_{j_1}, \eps_{i_2}+\eps_{j_2}\in \Delta(\fp^+)$
are incomparable if and only if either $i_1<i_2<j_2<j_1$ or $i_2<i_1<j_1<j_2$, it follows that
$\operatorname{width}(Y_\lambda)$ is the maximal nonnegative integer $m$ for which there exists a sequence 
of indices $1\leq i_1< i_2 < \cdots < i_m < j_m < \cdots < j_2< j_1\leq n$ such that  
$$
t_{i_1}+t_{j_1}\leq 0\ ,\  t_{i_2}+t_{j_2}\leq 0\ ,\  \ldots\  , \  t_{i_m}+t_{j_m}\leq 0.
$$
Since the sequence $(t_1,\ldots,t_n)$ is strictly decreasing, 
we may furthermore assume that $j_k=n-k+1$ for $k=1,\ldots, m$. This leads to the following simple algorithm 
to find $\operatorname{width}(Y_\lambda)$ with a running time that is linear in $n$, which is optimal.
Set $j_1=n$ and let $i_1$ be the smallest index such that $i_1<j_1$ and  $t_{i_1}+t_{j_1}\leq 0$. If no such $i_1$ exists, then  $m=0$.
Next set $j_2=n-1$ and let $i_2$ be the smallest index  such that $i_1<i_2<j_2$ and  $t_{i_2}+t_{j_2}\leq 0$. If no such $i_2$ exists, then $m=1$.
Continue doing this until it is no longer possible which happens after at most $[n/2]$ steps. For a concrete example, suppose
$$
\lambda+\rho=   (\underbracket[0pt]{\frac{25}{2},\frac{23}{2},\frac{19}{2}, \underbracket[0.5pt]{ \frac{15}{2}, \frac{13}{2}, \frac{11}{2}, \underbracket[0.5pt]{ \frac{9}{2}, \underbracket[0.5pt]{ -\frac{3}{2},-\frac{7}{2}} ,-\frac{9}{2}} ,-\frac{17}{2}}} ).
$$
Here brackets were drawn from the $j_1$-th entry to the $i_1$-th entry, from the $j_2$-th entry to the $i_2$-th entry, and from 
the $j_3$-th entry to the $i_3$-th entry.
We see that $\operatorname{width}(Y_\lambda)=3$ and the theorem says that $\AV(L(\lambda))=\overline{\mathcal{O}_7}$.
\end{example}
 
\subsection{Gelfand--Kirillov dimension}
Our main result has another interpretation in terms of the Gelfand--Kirillov (GK) dimension of Harish-Chandra modules.  The GK dimension of a Harish-Chandra module is equal to the dimension of its associated variety. 
Since the associated variety of a highest weight Harish-Chandra module $L(\lambda)$ is one of the varieties in the chain (\ref{eqn:chain}), it follows that the associated variety of $L(\lambda)$ is determined by the GK dimension of $L(\lambda)$.  The following theorem gives $\GKD(L(\gl))$ when $\gl$ is nonintegral.  One sees that in this case $\gD_\gl(\fp^+)$ is either empty or is isomorphic to a poset $\gD(\fp'^+)$ corresponding to a noncompact Hermitian symmetric space $G_\R'/K_\R'$, which in fact has simply laced root system.  This follows from Cartan's classification of Hermitian symmetric spaces and the work of Jakobsen \cite{Jak:83} on the posets of noncompact roots (see the diagrams in \cite[Appendix]{Jak:83}).

Let $\cO_k', k=0,1,\dots,r'$, be the $K'$-orbits in $\fp'^+$ (analogous to the chain described in \ref{eqn:chain}).  Set $\gd=\#(\gD^+)-\#(\gD_\gl^+)$, which equals $\#(\gD(\fp^+))-\#(\gD_\gl(\fp^+))$ since $\gl$ is $\gD(\fk)$-integral.
\begin{thm}\label{thm:GK}
Suppose $L(\lambda)$ is a highest weight Harish-Chandra module 
and assume that the highest weight $\lambda$ is nonintegral.  If $m=\width(Y_\gl)$, then 
\begin{equation*}
 \GKD(L(\gl))=\gd+\dim(\cO_m').
\end{equation*}
\end{thm}
It turns out that $\gD_\gl(\fp^+)$ is nonempty (for nonintegral $\gl$) only when $\gD$ is non-simply laced and $\gl$ is half-integral, see \cite[Lem. 3.17]{EHW:83}.   When $\gD_\gl(\fp^+)=\emptyset$, then $\cO_m'=\{0\}$ and $\gd=\#(\gD(\fp^+))=\dim(\fp^+)$, so $\AV(L(\gl))=\bar{\cO_r}=\fp^+$.

\subsection{About the proofs}
Theorems \ref{thm:main} and \ref{thm:GK} are proved in Section \ref{sec:proof}.  Theorem \ref{thm:main} is first proved when $\gl+\rho$ is regular integral; it suffices to assume the infinitesimal character is $\rho$.  The argument is to first observe that if $\width(Y_\gl)=m$, then $\AV(L(\gl))\supseteq\bar{\cO_m}$.  A counting argument is then used.  The general theory of  Harish-Chandra  cells gives a lower bound on the size of a cell in terms of the Springer correspondence.  This, along with the number of $L(\gl)$ (of infinitesimal character $\rho$) for which $\width(Y_\gl)=m$, completes the proof.  The number with a given width is calculated in Section \ref{sec:posets} using some combinatorics of the posets.  The necessary background on  Harish-Chandra  cells and some specific data on the Springer correspondence is given in Section \ref{sec:av-cells}. The case of arbitrary integral infinitesimal character follows from the translation principle, illustrating that our definition of the diagram $Y_\gl$ is `correct'.  This leaves the nonintegral case.  When $\gD$ is non-simply laced and $\gl$ is half-integral, there is a group $G_\R'$, as described above.  Lusztig's formula for the GK dimension proves Theorem \ref{thm:GK} and this reduces the proof of Theorem \ref{thm:main} to the integral case for $G_\R'$.   In the other nonintegral cases $L(\gl)$ is precisely a (full) generalized Verma module for the parabolic $\fq$;  the associated variety is then easily seen to be $\bar{\cO_m}=\fp^+$.

Another proof of Theorem \ref{thm:main} is given in Section \ref{sec:alt-proof}. In \cite{BX2}, Bai--Xiao--Xie used Robinson--Schensted insertion algorithm to compute the Gelfand--Kirillov dimensions of highest weight modules for classical type Lie algebras and cell decompositions of maximal parabolic subgroups of Hermitian type for exceptional type Lie algebras. Then, we translate these results into the world of antichains of posets, which gives us another proof for Theorem \ref{thm:main}.

\subsection{Connections with other work}
In 1992, Joseph \cite{Jos:92} described the associated varieties of unitary highest weight modules in a case-by-case fashion. Nishiyama--Ochiai--Taniguchi  \cite{NishiyamaOchiaiTaniguchi01} and Enright--Willenbring \cite{EW:04} independently gave some characterizations of the associated varieties of unitary highest weight  modules appearing in the dual pair settings. More recently, Bai--Hunziker \cite{BH:15} found  a uniform formula for the GK dimensions and associated varieties of all unitary highest weight  modules.  For $G_\R=Sp(2n,\R)$, Barchini--Zierau \cite{BZ17} computed the characteristic cycles of highest weight Harish-Chandra modules  for regular integral infinitesimal character and associated varieties are given implicitly by an
inductive procedure.
 
\section{Posets of positive noncompact roots}\label{sec:posets}

\subsection{\!\!}\label{ssec:pos-sys}
Let $\fg=\fp^- \oplus \fk\oplus \fp^+$ be the triangular decomposition 
corresponding to an irreducible Hermitian symmetric space (of noncompact type)   
as in the introduction. Therefore, $\fq=\fk\oplus \fp^+$ is a maximal  parabolic subalgebra
of $\fg$ with Levi subgroup $\fk$ and abelian nilradical $\fp^+\!$. We fix a Cartan subalgebra $\fh$ 
of $\fg$ that is contained in $\fk$ and a Borel subalgebra $\fb=\fh\oplus \fn$ of $\fg$ such that $\fb\subset\fq$. Let $\Delta$ and $\Delta(\fk)$ denote the root systems of $(\fg,\fh)$ and $(\fk,\fh)$, respectively, and set
$\Delta^+=\{ \alpha \in \Delta \mid \fg_\alpha \subseteq \fn\}$,  
$\Delta^+(\fk)=\Delta(\fk)\cap \Delta^+$, and $\Delta(\fp^+)=\{\ga\in\gD \mid  \fg_\ga\subset\fp^+\}=\Delta^+\setminus \Delta^+(\fk)$. 
The roots in $\Delta^+(\fk)$ are called the positive compact roots and the roots in 
$\Delta(\fp^+)$ are called the positive noncompact roots.  Let $\Pi$ denote the set of simple roots for $\Delta^+$. Exactly one simple root is in $\Delta(\fp^+)$.  As usual, $\rho $ is half the sum of the positive roots.  Denote by $W$ the Weyl group of $\gD$. 

Define $\gL^+(\fk):=\{\gl\in\fh^*\mid \IP{\gl+\rho}{\ga^\vee} \in\Z_{\geq 0}\text{ for all }\ga\in\gD^+(\fk)\}$.  Note that $\gL^+(\fk)$ is the set of highest weights of irreducible finite-dimensional  representations of $\fk$. We consider arbitrary $\gl\in\gL^+(\fk)$.  When $\gl$ is not integral we need to consider the integral root system $\gD_\gl=\{\ga\in\gD\mid \IP{\gl+\rho}{\ga^\vee}\in\Z\}$ and its Weyl group $W_\gl$.  

We view all sets of roots as partially ordered sets with the usual partial ordering, where 
$\alpha\leq \beta$ means that $\beta - \alpha$ is a sum of roots in $\gD^+$.
In what follows, the set of positive noncompact roots $\Delta(\fp^+)$, viewed as a poset, will play a fundamental role.

\begin{lemma}[{Jakobsen~\cite[Lem. 4.1]{Jak:83}}]
Let $\alpha \in \Delta(\fp^+)$, let  $\pi_{1},\ldots,\pi_{k}$
be any distinct elements of $\Pi \cap \Delta(\fk)$, and assume that $\alpha+\pi_{i}\in \Delta(\fp^+)$ for all $i=1,\ldots, k$. Then $k\leq 2$. Furthermore, if $k=2$, then 
$\pi_{1}$ and $\pi_{2}$ are orthogonal and  $\alpha+\pi_{1}+\pi_{2}\in \Delta(\fp^+)$. \hfill$\Box$
\end{lemma}

In light of this lemma, the Hasse diagram of $\Delta(\fp^+)$ is an upward planar graph 
of order dimension two and hence can be drawn on a two-dimensional orthogonal lattice that has been rotated by a 45-degree angle. That is, we draw the Hasse diagram as a graph with vertices labelled by the positive noncompact roots and a directed edge from $\gb$ to $\gb'$ when $\gb'-\gb$ is a simple (necessarily compact) root; the edge is sometimes labelled by the simple compact root.  The unique minimal element of the poset is the simple noncompact root (drawn at the bottom) and the unique maximal element, which we denote by $\theta$, is the highest root (drawn at the top). These diagrams appear in the appendix of Jakobsen's paper \cite{Jak:83}.  We have included the Hasse diagrams for the simply-laced types in the appendix.  

\subsection{Lower-order ideals}\label{ssec:loi}
A subset $Y\subseteq \Delta(\fp^+)$ is called a \emph{lower-order ideal} if, for  $\alpha \in \Delta(\fp^+)$ and $\beta \in Y$, $\alpha \leq \beta$ implies that $\alpha \in Y$. 

For $\lambda \in \Lambda^+(\fk)$,
define the \emph{diagram} of $\lambda$ is the set
\begin{equation}
Y_{\lambda} :=\{\alpha\in \Delta(\fp^+)\mid  \IP{\lambda +\rho}{\alpha^\vee}\in \Z_{\leq 0}\},
\end{equation}
viewed as a subposet of $\Delta(\fp^+)$.

\begin{lemma}\label{L: lower order ideals}
When $\gl$ is integral, the poset $Y_{\lambda}$ is a lower order ideal of $\Delta(\fp^+)$. 
\end{lemma}

\begin{proof}
Suppose $\alpha\in \Delta(\fp^+)$, $\beta\in Y_{\lambda}$ and $\alpha\leq \beta$.
Since $\alpha,\beta\in \Delta(\fp^+)$ and $\alpha\leq \beta$ it follows that
$\beta -\alpha$ is a sum of roots in $\Delta^+(\fk)$. Since $\lambda\in \Lambda^+(\fk)$,
it follows that $\alpha\in Y_{\lambda}$.
\end{proof}
It is easy to describe the lower order ideals in $\Delta(\fp^+)$. Define
\begin{equation}\label{eqn:script-W}
  \cW:=\{w\in W \mid -w\rho\text{ is $\Delta^+(\fk)$-dominant}\}.
\end{equation}
It follows that  $\cW=\{w\in W \mid \fn \cap \Ad(w) \fn \ \subseteq\ \fp^+\}$.   Note that the cardinality of $\cW$ is $\#\cW=\#(W/W(\fk))$, where $W(\fk)$ is the Weyl group of $\fk$.

\begin{prop}\label{prop:bijection}
There is a bijection 
\begin{equation}
\mathcal{W} \longleftrightarrow \{\mbox{lower order ideals of $\Delta(\fp^+)$}\} 
\end{equation}
given by $w\mapsto Y_{-w\rho-\rho}$. Furthermore, for any $w\in\mathcal{W}$, we have 
\begin{equation}
Y_{-w\rho-\rho}=\Delta(\fn\cap\Ad(w)\fn).
\end{equation}
\end{prop}

\begin{proof}
This is an easy consequence of \cite[\S3]{EHP:14}.  
\end{proof}

\begin{example}\label{ex:loi} Suppose $\fg_\R=\su(4,3)$
and $\lambda+\rho=-w\rho = (2,1,-1,-2 \mid 3,0,-3)$. Then
\begin{center}
\begin{pspicture}(-7,-0.7)(7,3.2)
\uput[r](-7,1){$\fn \cap \Ad(w) \fn =\left\{\displaystyle \left(\begin{array}{cccc|ccc} 
\cdot &\cdot&\cdot&\cdot&  * & \cdot  &\cdot\\ 
\cdot &\cdot&\cdot&\cdot&  * & \cdot &\cdot\\ 
\cdot &\cdot&\cdot&\cdot&  * & * &\cdot\\ 
\cdot &\cdot&\cdot&\cdot&  * & * & \cdot \\ 
\hline
\cdot &\cdot&\cdot&\cdot&  \cdot & \cdot & \cdot\\ 
\cdot &\cdot&\cdot&\cdot&  \cdot & \cdot & \cdot\\ 
\cdot &\cdot&\cdot&\cdot&  \cdot & \cdot & \cdot
\end{array}
\right)\right\} 
$}
\uput[r](1.3,1){and}
\uput[r](2.5,1){$Y_\lambda=$}
\cnode*(5.5,0){.07}{a0}
\uput[d](6.3,0){$\alpha_4=\eps_4-\eps_5.$}
\cnode*(5,.5){.07}{a1}
\cnode*(6,.5){.07}{a2}
\cnode*(4.5,1){.07}{a3}
\cnode*(5.5,1){.07}{a4}
\cnode*(4,1.5){.07}{a6}
\psset{linecolor=mygray}
\cnode*(6.5,1){.07}{a5}
\cnode*(5,1.5){.07}{a7}
\cnode*(6,1.5){.07}{a8}
\cnode*(4.5,2){.07}{a9}
\cnode*(5.5,2){.07}{a10}
\cnode*(5,2.5){.07}{a11}
\psset{linecolor=black}
\ncline{->}{a0}{a1}
\ncline{->}{a0}{a2}
\ncline{->}{a1}{a3}
\ncline{->}{a1}{a4}
\ncline{->}{a2}{a4}
\ncline{->}{a3}{a6}
\psset{linecolor=mygray}
\ncline{->}{a2}{a5}
\ncline{->}{a3}{a7}
\ncline{->}{a4}{a7}
\ncline{->}{a4}{a8}
\ncline{->}{a5}{a8}
\ncline{->}{a6}{a9}
\ncline{->}{a7}{a9}
\ncline{->}{a7}{a10}
\ncline{->}{a8}{a10}
\ncline{->}{a9}{a11}
\ncline{->}{a10}{a11}
\psset{linecolor=black}
\uput[d](5.1,.4){\scriptsize{3}}
\uput[d](4.6,.9){\scriptsize{2}}
\uput[d](4.1,1.4){\scriptsize{1}}
\uput[d](5.9,.4){\scriptsize{5}}
\uput[d](6.4,.9){\scriptsize{6}}
\end{pspicture}
\end{center}
\end{example}

It is useful to know how to associate an element of $\cW$ to a diagram $Y_\gl$.  Given $\gl\in\gL^+(\fk)$ that is integral (for $\gD$), we show how to determine $w\in \cW$ so that $Y_\gl=Y_{-w\rho-\rho}$.  Let $C^-:=\{y\in\fh^*\mid \IP{y}{\ga}<0,\text{ for all }\ga\in\gD^+\}$, the negative Weyl chamber.  When $\gl+\rho$ is regular there is a unique $w\in W$ such that $w^{-1}(\gl+\rho)\in C^-$.  This is the correct $w$ (since $w^{-1}(\gl+\rho)$ and $-\rho$ lie in the same Weyl chamber, which implies $\gl+\rho$ and $-w\rho$ are also in the same chamber).  However, when $\gl+\rho$ is singular the situation is more subtle. The following proposition shows how to determine the correct $w$.  The importance lies in Proposition \ref{prop:red}.

\begin{prop}\label{prop:correct-w} Suppose that $\gl\in\gL^+(\fk)$ and $\gl$ is integral.  Let $w\in W$ be a minimal length element of $W$ such that $w^{-1}(\gl+\rho)$ is in the closure of $C^-$. Then the following hold.
\begin{enumerate}
\item $w\in \cW$.
\item If $\ga\in\gD^+$ and $\IP{w^{-1}(\gl+\rho)}{\ga}=0$, then $w\ga>0$.
\item $Y_\gl=Y_{-w\rho-\rho}$.
\end{enumerate}
Such $w$ is unique.
\end{prop}
\begin{proof} Since $w^{-1}(\gl+\rho)$ and $-\rho$ both lie in the closure of $C^-$, $\gl+\rho$ and $-w\rho$ lie in the same (closure of) Weyl chamber.  Therefore, $-w\rho\in\gL^+(\fk)$, so $w\in\cW$.  Suppose that $\IP{w^{-1}(\gl+\rho)}{\ga}=0$ for some positive root $\ga$.  Then $(ws_\ga)^{-1}(\gl+\rho)=s_\ga(w^{-1}(\gl+\rho))=w^{-1}(\gl+\rho)$ (where $s_\ga$ is the reflection in $\ga$).  By the minimal length hypothesis for $w$, $\length(w)<\length(ws_\ga)$, which implies $w\ga>0$.  

The inclusion $Y_\gl\supseteq Y_{-w\rho-\rho}$ is clear since $\gl+\rho$ and $-w\rho$ lie in the same closure of Weyl chamber.  Therefore, suppose that $\gb\in Y_\gl$.  If $\IP{\gl+\rho}{\gb}\neq 0$, then $\IP{-w\rho}{\gb}<0$.  If $\IP{\gl+\rho}{\gb}=0$ there are two possibilities. Case (i): $w^{-1}\gb>0$.  Then $\IP{-w\rho}{\gb}=-\IP{\rho}{w^{-1}\gb}<0$, therefore $\gb\in Y_{-w\rho-\rho}$.  Case (ii): $w^{-1}\gb<0$.  Then $\ga:=-w^{-1}\gb\in\gD^+$ is orthogonal to $w^{-1}(\gl+\rho)$,  However, part (2) says that $-\gb=w\ga>0$, which is a contradiction, so Case (ii) does not occur.  Part (3)  is now proved.

Uniqueness follows from (3).
\end{proof}

\subsection{Antichains and width of lower-order ideals}\label{ssec:chains}
An \emph{antichain} in any poset $Y$ is a subset of $Y$ consisting of pairwise noncomparable elements.  The \emph{width} of the poset, denoted by $\width(Y)$, is the greatest cardinality of  antichains in $Y$. 

\begin{rem} Example \ref{ex:loi} illustrates the fact that for $\f{su}(p,q)$ the width of $Y_\lambda=Y_{-w\rho-\rho}$ is the maximal rank of a matrix in $\fn \cap \Ad(w) \fn$.  See the proof of Lemma \ref{lem:rank} for the analogous statement when $\fg_\R=\f{sp}(2n,\R)$.
\end{rem}

Recall that two roots $\gg_1,\gg_2\in\gD$ are called \emph{strongly orthogonal} if neither $\gg_1\pm\gg_2$ is a root. When $\gg_1,\gg_2\in\gD(\fp^+)$, $\gg_1+\gg_2$ is never a root (as $\fp^+$ is abelian), therefore, they are strongly orthogonal if and only if $\gg_1-\gg_2$ is a root.   Thus, $\gg_1,\gg_2$ are not comparable implies they are strongly orthogonal.  (The converse does not hold; for example, in $\f{so}(2,2n-1)$, $(\eps_1+\eps_n)-(\eps_1-\eps_n)$ is a sum of roots, but is not a root.)  The following lemma therefore holds.

\begin{lemma}\label{lem:width-so}
Every antichain of roots in $\Delta(\fp^+)$ is a set of pairwise strongly orthogonal roots.\hfill{$\square$}
\end{lemma}

\begin{cor}
For every $w\in \cW$, $\operatorname{width}(Y_{-w\rho-\rho})\leq r$, where $r$ is the real rank of $\fg_\R$.
\end{cor}
\begin{proof} 
It is well-known (see \cite{Johnson80}) that $r$ is the cardinality of a maximal set of mutually strongly orthogonal roots in $\gD(\fp^+)$.
\end{proof}

\subsection{The number of lower-order ideals of a given width}
The number of lower order ideals of width $m$ is determined in this section.  This is important for our counting arguments in Section~\ref{sec:proof}.

\begin{prop}\label{prop:width}
For $0\leq m \leq r$, the number of lower-order ideals in $\Delta(\fp^+)$ of given width $m$ is given by the last column of Table~\ref{tab:width}. 
\end{prop}

\begin{table}[H] 
\caption{The number of  lower-order ideals of $\Delta(\fp^+)$ of a given width}
\label{tab:width}
\centering
\renewcommand{\arraystretch}{1.5}
\setlength\tabcolsep{10pt}
\begin{tabular}{lll}
\toprule
$\fg_\R$ & $\card(\cW)$ & $\card\{w\in \cW \mid \operatorname{width}(Y_{-w\rho-\rho})=m\}$\\  
\midrule    
$\mathfrak{su}(p,q)$ &
${p+q \choose p}$& ${p+q \choose m} -{p+q \choose m-1}$\\[5pt] 
   $\mathfrak{sp}(2n,\R)$ & $2^{n}$ & ${n+1 \choose m}$ if $m<\frac{n+1}{2}$,\\
  & &  $\frac{1}{2}{n+1 \choose m}$ if $m=\frac{n+1}{2}$, and  \\
  & &  $0$ if $m>\frac{n+1}{2}$ \\[5pt]
   $\mathfrak{so}^*(2n)$ & $2^{n-1}$ & ${n \choose m}$ if $m<\frac{n}{2}$ and \\
  & &  $\frac{1}{2}{n \choose m}$ if $m=\frac{n}{2}$ \\
$\mathfrak{so}(2,2n-1)$ &  $2n$ & $1$, $2n-1$, $0$ \\[5pt]
$\mathfrak{so}(2,2n-2)$ & $2n$ & $1$, $n$, $n-1$\\ 
 $\mathfrak{e}_{6(-14)}$ & 27 & $1$, $6$, $20$\\[5pt] 
$\mathfrak{e}_{7(-25)}$ & 56 & $1$, $7$, $27$, $21$ \\ 
\bottomrule
\end{tabular}
\end{table}

The proof of this proposition is a  case-by-case argument.

\subsubsection{\it Proof for $\fg_\R=\su(p,q)$}  
Write $n=p+q$.  Take $\gD^+$ to be the usual positive system of roots for Type $A_{n-1}$, so $\gD^+=\{\eps_i-\eps_j\mid 1\leq i<j\leq n\}$ and $\gD(\fp^+)=\{\eps_i-\eps_j\mid 1\leq i\leq p<j\leq n\}$. 

Suppose that $\gb_1=\eps_{i_1}-\eps_{j_1}$ and $\gb_2=\eps_{i_2}-\eps_{j_2}$ are noncomparable roots in $\gD(\fp^+)$.  Then $i_1\leq p<j_1$ and $i_2\leq p< j_2$ and we may assume that $i_1<i_2$.  Then $\gb_1-\gb_2=(\eps_{i_1}-\eps_{i_2})+(\eps_{j_2}-\eps_{j_1})$.  The first is a positive root, so the second is necessarily a negative root (otherwise $\gb_1\geq\gb_2$), so $j_1<j_2$.  We may conclude that antichains take the form 
\begin{equation*}
\{\eps_{i_s}-\eps_{j_s}\mid s=1,2,\dots,m\}, \text{ where }i_1<i_2<\dots i_m\leq p<j_1<j_2<\dots<j_m.
\end{equation*}
If $Y$ has width $m$, then $Y$ contains a maximal antichain of the above form.  Since $Y$ is a lower order ideal we may subtract positive compact roots and stay in $Y$, therefore $Y$ contains the antichain
\begin{equation}\label{eqn:anti-supq}
\{\eps_{p-s+1}-\eps_{p+s}\mid s=1,2,\dots,m\},
\end{equation}
that is $i_1<i_2<\dots i_m= p<j_1<j_2<\dots<j_m$ are consecutive.  (It is useful to view this antichain in the Hasse diagram for $\gD(\fp^+)$: it is the horizontal row of nodes $m-1$ rows up from the minimal root. See the Appendix for the Hasse diagram for $\fs\fu(4,3)$.)  Note that $m$ is necessarily less than $r=\min\{p,q\}$.  We conclude that \emph{all} lower order ideals $Y$ of width $m$ have a maximal antichain in common, namely the one in (\ref{eqn:anti-supq}).

Let $\wp_{p,q}(m)=\#\{w\in\cW\mid \width(Y_{-w\rho-\rho})=m\}$.  (We interpret this to be $0$ if $p$ or $q$ is $0$.)  Notice that $\wp_{p,q}$ is symmetric in $p,q$. We assume $p\leq q$.

\noindent {\bf Claim}: $\wp_{p,q}(m)=\begin{cases}\wp_{p-1,q}(m)+\wp_{p,q-1}(m-1), &\text{if }m<p \\ \wp_{p,q-1}(p)+\wp_{p,q-1}(p-1), &\text{if }m=p. 
\end{cases}$

For our argument it is useful to keep in mind the Hasse diagram for $\gD(\fp^+)$.
Note that  $\rho=(\frac{n-1}{2},\frac{n-3}{2},\dots,-\frac{n-1}{2})$ and for $w\in \cW$ the entries of $-w\rho$ are the same as those in $\rho$ and the first $p$ entries decrease as do the last $q$ entries.  Therefore, $-w\rho$ takes the form 
\begin{center}
(a) $\,(\frac{n-1}{2},\dots\,|\,\dots) \text{ or (b) }\, (\dots\dots\,|\,\frac{n-1}{2},\dots).$
\end{center}
To prove the claim we consider the contributions to $\wp_{p,q}(m)$ from the two forms of $-w\rho$.    The contribution from (a) is $\wp_{p-1,q}(m)$, since no roots $\eps_1-\eps_j$ appear in $Y_{-w\rho-\rho}$.  The contribution from (b) is 
\begin{equation*}
\begin{cases}  \wp_{p,q-1}(p-1), &\text{ if }m<p  \\
\wp_{p,q-1}(p)+\wp_{p,q-1}(p-1), &\text{ if }m=p.
\end{cases}
\end{equation*}
This is easily understood from the Hasse diagram of $\gD(\fp^+)$.  The claim follows since $\wp_{p-1,q}(m)=0$ when $m=p$.

Now we show that for $p\leq q$
\begin{equation*}
\wp_{p,q}(m)=\binom{p+q}{m}-\binom{p+q}{m-1}, \text{ for }m=1,2,\dots,p,
\end{equation*}
using induction on $n=p+q$.  When $n=2$ this is immediate.  There are several cases.
\begin{enumerate}
\item[(i)] 
Case $m<p$:
\begin{align*}
\wp_{p,q}(m)&=\wp_{p-1,q}(m)+\wp_{p,q-1}(m-1), \text{ by the claim,} \\
  &=\left(\binom{n-1}{m}-\binom{n-1}{m-1}\right)+\left(\binom{n-1}{m-1}-\binom{n-1}{m-2}\right), \text{ by induction,}  \\
  &=\binom{n-1}{m}-\binom{n-1}{m-2}  \\
  &=\binom{n}{m}-\binom{n}{m-1}, \text{ an easily checked identity}.
\end{align*}
\item[(ii)]  Case $m=p<q$:
\begin{align*}
\wp_{p,q}(p)&=\wp_{p,q-1}(p)+\wp_{p,q-1}(p-1), \text{ by the claim,} \\
  &=\left(\binom{n-1}{p}-\binom{n-1}{p-1}\right)+\left(\binom{n-1}{p-1}-\binom{n-1}{p-2}\right), \text{ by induction,}  \\
  &=\binom{n-1}{p}-\binom{n-1}{p-2}  \\
  &=\binom{n}{p}-\binom{n}{p-1}, \text{ same identity}.
\end{align*}   
\item[(iii)]  Case $m=p=q$:
\begin{align*}
\wp_{p,p}(p)&=\wp_{p,p-1}(p)+\wp_{p,p-1}(p-1), \text{ by the claim,} \\
  &=\wp_{p,p-1}(p)+\wp_{p-1,p}(p-1), \text{ by symmetry,} \\
  &=\binom{2p-1}{p-1}-\binom{2p-1}{p-2}, \text{ first term is $0$,}  \\
 &=\binom{2p-1}{p}-\binom{2p-1}{p-2}  \\  
 &=\binom{2p}{p}-\binom{2p}{p-1}, \text{ same identity}.
\end{align*}  
\end{enumerate}

\subsubsection{\it Proof for $\fg_\R=\so^*(2n)$}\label{ssec:so*} This case is is quite similar to the $\f{su}(p,q)$ case above. The set of positive compact roots is $\gD^+(\fk)=\{\ep_i-\ep_j \mid 1\leq i< j\leq n\}$ and the poset of positive noncompact roots is $\gD(\fp^+)=\{\ep_i+\ep_j \mid 1\leq i< j\leq n\}$.  Suppose $\gg_1=\ep_{i_1}+\ep_{j_1}$ and $\gg_2=\ep_{i_2}+\ep_{j_2}$ are two noncomparable roots in $\gD(\fp^+)$.  We may assume that $i_1<i_2, i_1<j_1$ and $i_2<j_2$.  Then in the expression
$$\gg_1-\gg_2=(\ep_{i_1}-\ep_{i_2}) + (\ep_{j_1}-\ep_{j_2})
$$
the first term  is a positive root, therefore $\gg_1,\gg_2$ are noncomparable if and only if the second term is a negative root.  This holds if and only if $j_2<j_1$ (so $i_1<i_2<j_2<j_1$).  We may conclude from this that the antichains of length $m$ in any lower order ideal  are 
$$
\{\ep_{i_s}+\ep_{j_s}\mid s=1,\dots,m\}, \text{ where }  i_1<\dots<i_m<j_m<\dots,<j_1.
$$
For $w\in\cW$,  $-w\rho$ has decreasing entries, so if $\width(Y_{-w\rho-\rho})=m$ then we may assume that $i_1<\dots<i_m<j_m<\dots,j_1$ are \emph{consecutive}  ending in $j_1=n$:
$\{\ep_{n-2m+s} + \ep_{n+1-s}\mid  s=1,\dots,m\}.$
Therefore, whenever a lower order ideal has width $m$ it contains this maximal antichain, which appears in the Hasse diagram as the horizontal row of nodes appearing $2m-2$ rows up from the bottom.
For $n=6$, the Hasse diagram is shown in the appendix.
 
Let $\wp_n(m)$ be the number of $w\in \cW$ such that $\width(Y_{-w\rho-\rho})=m$.  Since $\rank_\R(\f{so}^*(2n))=\lfloor \frac{n}{2}\rfloor$, $\wp_n(m)=0$ when $m>\lfloor \frac{n}{2}\rfloor$.

\begin{claim} We have
$$
\wp_n(m)=
  \begin{cases}  \binom{n}{m}, &\text{if } m<\frac{n}{2}  \\
                 \frac12\binom{n}{n/2}, &\text{if }m= \frac{n}{2}.
  \end{cases}
$$  
\end{claim}

\begin{proof}[Proof of claim]
Since $\rho=(n-1,\dots 2,1,0)$ and $-w\rho$ has decreasing entries,  $-w\rho$ has the form   
\begin{equation*} 
\text{(a) }(n-1,\dots\dots)\quad \text{ or\quad  (b) } (\dots\dots,-(n-1)).
\end{equation*}
As in the $\f{su}(p,q)$ case we count the contributions to $\wp_n(m)$ from the two cases. It is again useful to consider the Hasse diagram.  The contribution from case (a) is $\wp_{n-1}(m)$, since $\eps_1+\eps_j$ is never in $Y_{-w\rho-\rho}$.  The contribution from case (b) is $\wp_{n-1}(m-1)$, when $m\neq\frac{n-1}{2}$, and is $\wp_{n-1}(m-1)+\wp_{n-1}(m)$, when $m=\frac{n-1}{2}$. Thus, we have
\begin{equation*}
\wp_n(m)=\begin{cases}
\wp_{n-1}(m)+\wp_{n-1}(m-1), & \text{if }m\neq\frac{n-1}{2}  \\
\wp_{n-1}(m)+\wp_{n-1}(m-1)+\wp_{n-1}(m), & \text{if }m=\frac{n-1}{2}.  
\end{cases}
\end{equation*}
We now prove the claim by induction on $n$.  There are three cases.
\begin{itemize}
\item[(i)] $\displaystyle m<\frac{n-1}{2}$:
$$
\wp_{n}(m)=\wp_{n-1}(m)+\wp_{n-1}(m-1)=\binom{n-1}{m}+\binom{n-1}{m-1}=\binom{n}{m}.
$$

\item[(ii)] $\displaystyle m=\frac{n-1}{2}$:
\begin{align*}
\wp_{n}(m)&=2\wp_{n-1}\left(\frac{n-1}{2}\right)+\wp_{n-1}\left(\frac{n-1}{2}-1\right)\\
&=\binom{n-1}{(n-1)/2}+\binom{n-1}{(n-3)/2}=\binom{n}{(n-1)/2}.
\end{align*}

\item[(iii)] 
$\displaystyle m=\frac{n}{2}$:
$$
\wp_{n}(m)=\wp_{n-1}\left(\frac{n}{2}\right)+\wp_{n-1}\left(\frac{n}{2}-1\right)=0+\wp_{n-1}\left(\frac{n}{2}-1\right)=\binom{n-1}{n/2-1}=\frac12\binom{n}{n/2}.
$$
\end{itemize}
\end{proof}

\subsubsection{\it Proof for $\fg_\R=\sp(2n,\R)$} \label{sec:width-sp}
The poset $\Delta(\fp^+)$ is isomorphic to $\Delta(\fp'^+)$,
where $\fg'_\R=\so^*(2n+2)$. Therefore, this case is already done.  Note that the real rank of $\sp(2n,\R)$ is $n$ whereas the real rank of $\so^*(2n+2)$ is $\lfloor \frac{n+1}{2}\rfloor$. This explains the entry $0$ for $m>\frac{n+1}{2}$
in the third column of the Table~\ref{tab:width}. 

Using the isomorphism with the poset $\Delta(\fp'^+)$ and the description of antichains given in \S\ref{ssec:so*}, we see that if $\width(Y)=m$, then there is an maximal antichain having one of  the forms
\begin{align*}
&\left\{\ep_{n-m+1-s}+\ep_{n-m+1+s}\mid s=1,\dots,m-1\right\}\cup\left\{2\ep_{n-m+1}\right\}\text{ or }\\
&\left\{\ep_{n-m+1-s}+\ep_{n-m+s}\mid s=1,\dots,m\right\}.
\end{align*}

\subsubsection{\it Proof for $\fg_\R=\so(2,2n-1)$}
The poset of $\Delta(\fp^+)$ is a linear graph.  The result is immediate because $\#\gD(\fp^+)=2n-1$.

\subsubsection{\it Proof for $\fg_\R=\so(2,2n-2)$} 
The poset $\gD(\fp^+)$ in this case is quite simple and is illustrated below. Noting that the only antichain of width $2$ is $\{\eps_1-\eps_2,\eps_1+\eps_2\}$, the lower order ideals of widths $0,1$ and $2$ are easily counted.  We get $\wp_n(0)=1, \wp_n(1)=n,$ and $\wp_n(2)=n-1$.
For $n=6$, the Hasse diagram of $\Delta(\fp^+)$ is shown in the appendix.

\subsubsection{\it Proof for 
 $\fg_\R=\mathfrak{e}_{6(-14)} $}
The poset $\Delta(\fp^+)$ for $\fg_\R=\mathfrak{e}_{6(-14)} $ is shown in the Appendix. A lower-order ideal $Y\subseteq \Delta(\fp^+)$ has width $\geq 1$ if and only if $\alpha_1\in \Delta(\fp^+)$
and width $\geq 2$ if and only if $\beta_1,\beta_2\in \Delta(\fp^+)$.
A quick inspection shows that there are exactly $6$ lower-order ideals of $\Delta(\fp^+)$ that contain $\alpha_1$, but not both $\beta_1$ and $\beta_2$.
Thus, there are exactly $6$ posets  of $\Delta(\fp^+)$ that have width $1$. Since the total number of lower-order ideals of $\Delta(\fp^+)$ is $\card(\cW)=27$,
the number of lower-order ideals of $\Delta(\fp^+)$ width $2$ is $27-1-6=20$. 

\subsubsection{\it Proof for 
 $\fg_\R=\mathfrak{e}_{7(-25)} $}
The poset $\Delta(\fp^+)$ for $\fg_\R=\mathfrak{e}_{7(-25)} $ is shown in the Appendix.   A lower-order ideal $Y\subseteq \Delta(\fp^+)$ has width $\geq 1$ if and only if $\alpha_7\in \Delta(\fp^+)$,
width $\geq 2$ if and only if $\beta_1,\beta_2\in \Delta(\fp^+)$,
and width $\geq 3$ if and only if $\gamma_1,\gamma_2,\gamma_3\in \Delta(\fp^+)$.
Again, a quick inspection shows that there are exactly $7$ lower-order ideals of $\Delta(\fp^+)$ that contain $\alpha_1$, but not both $\beta_1$ and $\beta_2$. A more tedious inspection shows that there are exactly $27$ lower-order ideals of $\Delta(\fp^+)$
that contain $\beta_1$ and $\beta_2$ but not all three roots $\gamma_1$, $\gamma_2$, and $\gamma_3$.
Thus, there are exactly $7$ lower-order ideals  of $\Delta(\fp^+)$ of width $1$,  exactly $27$ lower-order ideals  of width $2$,
and $56-1-7-27=21$ lower-order ideals of  width $3$.

This concludes the proof of Proposition~\ref{prop:width}. \hfill $\Box$

\subsection{Half-integral case}\label{ssec:half-2} We need to consider highest weight Harish-Chandra modules $L(\gl)$ when $\gl$ is   half-integral.  As we shall see, all the information needed is obtained by reducing to the integral case.

Let $\gl\in\gL^+(\fk)$ be half-integral.  This means that $\IP{\gl+\rho}{\gb^\vee}\in \frac12+\Z$, for some $\gb\in\gD(\fp^+)$.  We consider the integral root system
\begin{equation*}
\gD_\gl=\{\ga\in\gD\mid \IP{\gl+\rho}{\ga^\vee}\in\Z\}.
\end{equation*}
Let $W_\gl$ denote the Weyl group of $\gD_\gl$ and  fix the positive system $\gD_\gl^+:=\gD^+\cap\gD_\gl$.  Let $\fg_\gl$ be the complex semisimple Lie algebra with root system $\gD_\gl$.  Then there is a triangular decomposition
\begin{equation}\label{eqn:tri-lamda}
\fg_\gl=\fp_\gl^-+\fk+\fp_\gl^+,
\end{equation}
where $\gD(\fp_\gl^\pm):=\gD(\fp^\pm)\cap\gD_\gl$.  Note that $\gD(\fk)\subset\gD_\gl$, since $\gl$ is $\gD(\fk)$-integral, so $\fk_\gl\simeq\fk$.   Therefore, we may apply  \S\ref{ssec:loi} and \S\ref{ssec:chains} to $\gD(\fp_\gl^+)\subset\gD_\gl^+$ in place of $\gD(\fp^+)\subset\gD^+$.   In particular, $Y_\gl$ is a lower order ideal in $\gD_\gl$ and defining $\cW_\gl$ in $W_\gl$ there is a bijection $\cW_\gl\leftrightarrow\{\text{lower order ideals in }\gD(\fp_\gl^+)\}$.  Note that for $\ga,\gb\in\gD(\fp_\gl^+)$, $\ga\leq\gb$ is the \emph{same} in $\gD$ (with respect to $\gD^+$) as in $\gD_\gl$ (with respect to $\gD_\gl^+$), since $\gb-\ga$ is a sum of compact roots and $\gD(\fk)\subset \gD_\gl$.   Therefore, $Y_\gl$ is the same whether defined in $\gD$ or in $\gD_\gl$.

There are two examples we return to in \S\ref{ssec:half}.  
\begin{itemize}
\item[(a)] $\fg_\R=\f{sp}(2n,\R),n\geq 2$.  If $\gl=(\gl_1,\dots,\gl_n)$ is half-integral, then  $\gl_j\in\frac12+\Z$, all $j$.  Therefore, the integral root system is
\begin{equation*}
 \gD_\gl=\{\eps_j\pm\eps_k\mid 1\leq j<k\leq n\},
\end{equation*}
a root system of type $D_n$. Observe that the complex Lie algebra $\fg_\gl$ corresponding to $\gD_\gl$ is of type $D_n$ and has a real form $\fg_{\gl,\R}\simeq \f{so}^*(2n)$ of Hermitian type.   The set of positive noncompact roots is $\gD(\fp_\gl^+)=\gD(\fp^+)\cap\gD_\gl=\{\eps_j+\eps_k:1\leq j<k\leq n\}$.
\item[(b)] $\fg=\f{so}(2,2n-1), n\geq 2$.  Then 
\begin{equation*}
 \gD_\gl=\{\pm\eps_1\}\cup\gD(\fk),
\end{equation*}
a root system of type $A_1\times B_{n-1}$.  The corresponding Lie Algebra $\fg_\gl$ has a real form $\fg_{\gl,\R}\simeq \f{sl}(2,\R)\times\f{so}(2n-1)$ of Hermitian type.
\end{itemize}

\section{Associated varieties and highest weight Harish-Chandra  modules}\label{sec:av-cells}  In this section we review some background information  on associated varieties, cells and the Springer correspondence.  

\subsection{Associated varieties and cells for $(\fg,K)$-modules} \label{ssec:AV}

Assume that $G_\R$ is a connected semisimple Lie group with finite center.  Let $\fg_\R=\fk_\R\oplus\fp_\R$ be a Cartan decomposition; since $G_\R$ has finite center, there is a maximal compact subgroup $K_\R$ with Lie algebra $\fk_\R$.  Let $\fg=\fk\oplus\fp$ be the complexified Cartan decomposition and $K$ the connected complex algebraic group containing $K_\R$ with Lie algebra $\fk$.   Any irreducible $K_\R$ representation extends to an algebraic representation of $K$.  Therefore, $(\fg,K)$ is a pair as in \cite[\S 1]{Vogan91} and the Harish-Chandra module of an admissible representation of $G_\R$ is a $(\fg,K)$-module.  The action of the algebraic group $K$ is needed for definition of associated variety.  We follow the definition of associated variety of a finitely generated $(\fg,K)$-module as given in \cite{Vogan91}. Note that we are including finite covers of linear groups, such as the metaplectic cover of $Sp(2n,\R)$ which occurs for half integral highest weights.  However, groups with infinite center, such as the universal cover of $G_\R=Sp(2n,\R)$, are not included in our discussion.  

Suppose that $(\fg,K)$ is a pair as above and  $X$ is  a $(\fg,K)$-module of finite length.   The associated variety of $X$ is defined by starting with a $K$-invariant `good' filtration $\{X_n\}$ of $X$.  Then $\Gr X$ is a module for $\Gr \cU(\fg)\simeq S(\fg)\simeq P(\fg^*)$, and by $K$-invariance of the filtration, is a module for $S(\fg/\fk)$.  The associated variety of $X$ is defined as the support of $\Gr X$:
\begin{equation*}
\AV(X):=\supp(\Gr X)\subseteq (\fg/\fk)^*.
\end{equation*}
Identifying $(\fg/\fk)^*$ with $\fp$ by the Killing form, $\AV(X)$ is a $K$-stable subvariety of $\fp$.  In fact, $\AV(X)\subseteq\fp\cap\cN$, where $\cN$ is the nilpotent cone in $\fg$.  Therefore,
\begin{equation}\label{eqn:av}
 \AV(X)=\cup\, \bar{\cO_\ga},
\end{equation}
a union of closures of nilpotent $K$-orbits in $\fp$ and we write this union in a way that no $\bar{\cO}_\ga$ is contained in any other.

When $X$ is irreducible, the $\cO_\ga$ occurring in (\ref{eqn:av}) are all equidimensional and all have the same $G$-saturation, that is, there is a complex nilpotent orbit $\cO^\C$ such that 
\begin{equation}\label{eqn:av-ann}
 \cO^\C=G\cdot\cO_\ga, \text{ for all }\ga.
\end{equation}
It is a fact that  $\bar{\cO^\C}$ is the associated variety of the annihilator of $X$.  See \cite{Vogan91} for details.
 
 Suppose $X$ and $Y$ are irreducible Harish-Chandra modules having infinitesimal character of Harish-Chandra parameter $\gL\in\fh^*$.  We say that $X\lesssim Y$ when $X$ occurs as a subquotient of  $Y\otimes F$, with $F$ some finite-dimensional $G$-subrepresentation of the tensor algebra of $\fg$.  We say  $X\sim Y$ when  $X\lesssim  Y$ and $Y\lesssim X$; this generates an equivalence relation on the set of irreducible representations of infinitesimal character $\gL$.  The \emph{Harish-Chandra cells} are defined to be the equivalence classes.  Since $X\lesssim Y$ implies $\AV(X)\subseteq\AV(Y)$, we see that the associated variety is constant on any given cell.

Now assume that $\gL$ is regular and integral and assume also that $G_\R$ is linear.  Then the coherent continuation representation of the Weyl group $W$ of $\fg$ is defined on the Grothendieck group  of Harish-Chandra modules of infinitesimal character $\gL$.  This gives a representation of $W$ for each cell as follows.   Fix an irreducible $X_0$ (infinitesimal character $\gL$) and let $\fC$ be the cell of $X_0$, that is, the cell containing $X_0$.  Set
\begin{equation*}
W_{\fC}:=\spa_\C\{X\mid X\lesssim X_0\}\text{ and }
U_{\fC}:=\spa_\C\{X\mid X\lesssim X_0\text{ and }X\nsim X_0\}
\end{equation*}
Then both $W_\fC$ and $U_\fC$ are $W$-stable under coherent continuation.  The \emph{cell representation} for $\fC$ is the quotient $V_\fC:=W_\fC/U_\fC$.  We may view 
\begin{equation*}
 V_\fC=\sum_{X\in\fC}\C\cdot X
\end{equation*}
(but keeping in mind that it is really a quotient).  See \cite{BarbaschVogan83} for cells and \cite[Ch.~7]{Vogan81} for coherent continuation.

Much is known about cells and the cell representations.  We have already mentioned that the associated variety is constant on cells, as is the associated variety of annihilator.  An important property of the cell representations is that if $\bar{\cO^\C}$ is the associated variety of annihilator, as in (\ref{eqn:av-ann}), then the $W$ representation $\pi(\cO^\C,1)$ corresponding to $\cO^\C$ under the Springer correspondence occurs in $V_\fC$ (\cite[Cor.~14.11]{Vogan82-IC4}) and is a special representation (in the sense of Lusztig). 

What makes this useful for us is that there are easy combinatorial procedures known (for the classical groups) and tables (for the exceptional groups) that (1) give the Springer correspondence and (2) determine which $W$-representations are special.  See \cite[\S13.1-13.3]{Carter85}.  The Atlas of Lie Groups software (\cite{atlas}) also gives this information for Lie algebras of rank up to at least 8.  

Most of the relevant data for us is contained in Table \ref{tab:springer}.

\subsection{Highest weight Harish-Chandra  modules}\label{ssec:hwhc}
A \emph{highest weight Harish-Chandra module} for our group $G_\R$ (having finite center) is a $(\fg,K)$-module (for the pair described in    \ref{ssec:AV}) having a vector $x_\gl^+\in X$ satisfying
\begin{enumerate}
\item[(i)] $x_\gl^+$ is a weight vector for $\fh$, a compact Cartan subalgebra of $\fg$, of weight $\gl\in\fh^*$, 
\item[(ii)] $\fn\cdot x_\gl^+=0$, where $\fn$ is the nilradical of some Borel subalgebra $\fb=\fh\oplus\fn\subset\fg$, and
\item[(iii)] $X=\cU(\fg)\cdot x_\gl^+$.
\end{enumerate}
Suppose $X$ is an infinite-dimensional irreducible highest weight Harish-Chandra module for $G_\R$.  Then, by \cite[\S~1-2]{HC55}, $G_\R$ is of Hermitian type.  In this case, there is a triangular decomposition  $\fg=\fp^-+\fk+\fp^+$ as in  \S\ref{ssec:pos-sys}.  It is also shown in \cite{HC55} that $\fh$ and $\fb$ in (i) and (ii) are given by our choices in \S\ref{ssec:pos-sys}.  Therefore, $\gl\in\gL^+(\fk)$ and $X$ contains the $K$-type $F(\gl)$ having highest weight $\gl$.  Furthermore,
\begin{equation*}
 X=\cU(\fg)x_\gl^+=\cU(\fp^-)F(\gl). 
\end{equation*}
We conclude that $X$ is the unique irreducible quotient of the generalized Verma module $\cU(\fg)\!\!\underset{\;\fk+\fp^+}{\otimes}\!F(\gl)$.  Therefore, $X$ is also the irreducible quotient $L(\gl)$ of the full Verma module
\begin{equation}\label{eqn:verma}
  M(\gl)=\cU(\fg)\underset{\fb}{\otimes}\C_\gl.
\end{equation}
The infinitesimal character is $\gl+\rho$.

We recall that if $X=L(\gl)$ is the irreducible quotient of a Verma module for our choice of $\fb$ (as in (\ref{eqn:verma})), then 
$\gl\in\gL^+(\fk)$ and $\IP{\gl+\rho}{\gb}\in\R$, for all $\gb\in\gD$ (see \cite[p.~100]{EHW:83}).

Our condition that $G_\R$ has finite center implies, by \cite{Lepowsky73}, that $X=L(\lambda)$ is the Harish-Chandra module of a continuous representation of $G_\R$.

\subsection{Associated varieties of highest weight Harish-Chandra  modules}
Suppose $X$ is any irreducible highest weight Harish-Chandra  module for $G_\R$ of highest weight $\gl$. Then (as in \S\ref{ssec:hwhc}) $X=\cU(\fp^-)F(\gl)$.  Using the usual filtration of the enveloping algebra we obtain a filtration
\begin{equation}\label{eqn:filtration}
 X_n:=\cU_n(\fp^-)F(\gl).
\end{equation} 
One easily checks that it is a good filtration and is $K$-invariant.  It is also $\fk+\fp_+$-invariant (so also $\fb$-invariant).  We conclude that $\AV(X)\subseteq \left(\fg/\fk+\fp^+\right)^*\simeq\fp^+$.

It is well-known that the $K$-orbits in $\fp^+$ have a particularly nice form.  Letting $r=\rank_\R(\fg_\R)$, there are $r+1$ such orbits.  They may be written as follows.  Let $\{\gg_1,\dots,\gg_r\}$ be a maximal set of mutually strongly orthogonal long roots in $\fp^+$ (necessarily of cardinality $r$).  Any two sets of mutually strongly orthogonal long roots having the same cardinality are conjugate  under the Weyl group $W(\fk)$. The $K$-orbits in $\fp^+$ are
\begin{equation}\label{eqn:K-orbits}
\begin{split}  
  &\cO_0=\{0\}\text{ and}   \\
  &\cO_k=K\cdot(X_{\gg_1}+\cdots+X_{\gg_k}), \, k=1,2,\dots,r,
\end{split}
\end{equation}
where  $X_\gg$ is a root vector for $\gg$.  Note that the choice  (and numbering) of strongly orthogonal roots doesn't matter, by the $W(\fk)$-conjugacy mentioned above.  These orbits satisfy  the inclusions 
\begin{equation*}
  \bar{\cO_k}\subseteq \bar{\cO_{k+1}}.
\end{equation*}
We conclude that the associated variety of any highest weight Harish-Chandra  module is the closure of exactly one $\cO_k$.

The following table collects relevant data on the $K$-orbits in $\fp^+$.

\begin{table}[H] 
\caption{Data on orbits and the Springer correspondence}
\label{tab:springer}
\renewcommand{\arraystretch}{1.25}
\setlength\tabcolsep{6pt}
\begin{tabular}{lllll}
\toprule
$\fg_\R$ & $\dim(\O_k)$ & label of $\O_k^\C$ & special? & $\dim(\pi(\O_k^\C,1))$\\  
\midrule    
$\mathfrak{su}(p,q)$ & $k(p+q-k)$ & $[2^k,1^{p+q-2k}]$ & yes & ${p+q \choose p} - {p+q \choose k-1}$ \\[5pt]
\midrule  
$\mathfrak{sp}(2n,\R)$ & $\frac{1}{2}k(2n-k+1)$ & $[2^k,1^{2n-2k}]$ & yes  & ${n\choose \frac{k}{2}}$ if $k$ is even\\
 &   &   & no  & ${n\choose \frac{k-1}{2}}$ if $k$ is odd, $k<n$\\[5pt]
 &   &  & yes  & ${n\choose \frac{n-1}{2}}$ if $k=n$ is odd, $k=n$\\[5pt]
\midrule  
$\mathfrak{so}^*(2n)$ & $k(2n-2k-1)$ & $[2^{2k},1^{2n-4k}]$ & yes & ${n \choose k}$ if $k\not=\frac{n}{2}$\\ 
  &   &   & yes & $\frac{1}{2}{n \choose \frac{n}{2}}$ if $k=\frac{n}{2}$\\ 
\midrule  
$\mathfrak{so}(2,2n-1)$ &  $2n-2$ & $[2^2,1^{2n-3}]$ & no & $n-1$ \\ 
  &  $2n-1$ & $[3,1^{2n-2}]$ & yes & $n$ \\[5pt]
\midrule  
$\mathfrak{so}(2,2n-2)$ & $2n-3$ & $[2^2,1^{2n-4}]$ & yes & $n$\\ 
             & $2n-2$ & $[3,1^{2n-3}]$ & yes & $n-1$\\ 
\midrule  
 $\mathfrak{e}_{6(-14)}$ &  $0$ & $0$ & yes & $1$\\ 
    &  $11$ & $A_1$ & yes & $6$\\  
    &  $16$ & $2A_1$ & yes & $20$ \\[5pt]  
\midrule  
$\mathfrak{e}_{7(-25)}$ & $0$ & $0$ & yes & $1$\\
  & $17$ & $A_1$ & yes & $7$\\
  & $26$ & $2A_1$ & yes & $27$\\
  & $27$ & $(3A_1)''$ & yes & $21$\\
\bottomrule
\end{tabular}
\end{table}

\begin{rem}
By Jantzen's irreducibility criterion or by \cite[Lem. 3.17]{EHW:83}, the generalized Verma module $\cU(\fg)\!\underset{\;\fk+\fp^+}{\otimes}F(\gl)$ can be reducible only when $\gl$ is integral or half-integral.  In the integral case we may assume that $G_\R$ is linear and in the half-integral case we may assume that $G_\R$ is the double cover of a linear group.
\end{rem}

\begin{rem}\label{rem:ex-av}
If $X$ is an \emph{irreducible} generalized Verma module $\cU(\fg)\!\!\underset{\;\fk+\fp^+}{\otimes}\!F(\gl)$ (as in the previous remark as well as in the case of holomorphic discrete series representations), then it follows easily from the definitions that $\AV(X)=\bar{\cO_r}=\fp^+$.  One may also see from the definitions that if one member of a cell is a highest weight Harish-Chandra module, then the cell consists entirely of highest weight Harish-Chandra modules.  Also, if the associated variety of a Harish-Chandra module $X$ is contained in $\fp^+$, then $X$ is a highest weight module (\cite[Prop.~B1]{BarchiniZierau16}).  If $\AV(X)=\cO_0=\{0\}$, then $X$ is finite-dimensional, so $\fC_0=\{\C\}$ is the only cell (for infinitesimal character $\rho$) with associated variety $\cO_0$,
\end{rem}

See \cite{NishiyamaOchiaiTaniguchi01} for more on associated varieties of highest weight Harish-Chandra modules.

Now consider any highest weight module $L(\gl)$.  We may apply the the theory of associated varieties of highest weight modules (as in \cite{BorhoBrylinski85} and \cite{Jos:92}).  This tells us that the associated varieties are unions of orbital varieties, which we write as
\begin{equation*}
  \nu(w):=\bar{B\cdot(\fn\cap \Ad(w)\fn)},\quad w\in W.
\end{equation*}
As in \cite{BorhoBrylinski85} we write $L_w=L(-w\rho-\rho), w\in W$.  Then the support of $L_w$ is the Schubert variety $Z_w=\bar{B\cdot w\fb}$ in the flag variety $\fB$ of $G$.  By \cite[Prop. 6.11 ]{BorhoBrylinski85} the associated variety of $L_w$ contains the moment map image  $\mu(\bar{T_\fB^*Z_w})=\nu(w)$, so
\begin{equation}\label{eqn:avBB}
  \AV(L_w)\supset \nu(w).
\end{equation}
(However, the associated variety is sometimes larger and of greater dimension.)  

Our first goal is a formula for associated varieties when $L_w$ is a Harish-Chandra module, that is, when $w\in\cW$.  Our description of $\AV(L_w)$ may be given in terms of some $\cO_k$ \emph{or} in terms of orbital varieties.  Theorem \ref{thm:main} expresses the associated variety in terms of $K$-orbits in $\fp^+$, however the orbital variety picture will allow us to use the combinatorics of posets discussed in Section \ref{sec:posets}.  
The following lemma puts this into perspective.
\begin{lemma}[see \cite{BarchiniZierau16}] 
\begin{enumerate}
\item[(1)] For each $j=0,1,\dots,r,$  we have $\bar{\cO_j}=\nu(w)$ for some $w\in \cW$.  
\item[(2)] If $\nu(w)\subseteq\fp^+$, then $\nu(w)=\bar{\cO_j}$,  for some $j=0,1,\dots,r$, and $w$ is necessarily in $\cW$.
\end{enumerate}
\end{lemma}

The proof of Theorem \ref{thm:main} will be largely a counting argument using Tables \ref{tab:width} and \ref{tab:springer} along with the following important fact:
\begin{equation}\label{eqn:key}
  \width(\gD(\fn\cap \Ad(w)\fn))=m\implies \AV(L_w)\supset \bar{\cO_m}.
\end{equation}  
This fact follows from (\ref{eqn:avBB}), (\ref{lem:width-so}) and (\ref{eqn:K-orbits}).

\section{Proof of Theorems \ref{thm:main} and \ref{thm:GK}}\label{sec:proof}
Theorem \ref{thm:main} is proved first for integral infinitesimal character.   We begin by showing that it suffices to prove the theorem for infinitesimal character $\rho$.  The proof for $X=L_w  ~(w\in \cW)$ is then carried out separately for the simply laced  and non-simply laced cases.  Then we consider the half-integral case.  We first prove Theorem \ref{thm:GK}, then  reduce the half-integral case to the integral case.

\subsection{The translation principle}\label{sec:rho}  The reduction of the proof of Theorem \ref{thm:main} from integral infinitesimal character to infinitesimal character $\rho$ is accomplished using the translation principle.  This principle is well-known; we use \cite[Ch.~7]{Humphreys08} as a convenient reference; see also \cite{Jantzen79}.  A translation functor $T$ applied to $X$ is tensoring by a finite-dimensional $\fg$-representation followed by projection to the constituents of a generalized infinitesimal character.  Therefore, one has $\AV(T(X))\subseteq\AV(X)$.  The following proposition is essentially contained in \cite[Cor.~3.3]{BX}; we give the proof as it plays a key role for us.

\begin{prop}\label{prop:red}  Let $L(\gl)$ be a highest weight Harish-Chandra module having integral infinitesimal character $\gl+\rho$.  Choose $w\in \cW$ such that $Y_\gl=Y_{-w\rho-\rho}$ (as in Prop.~\ref{prop:bijection}).  The following hold.
\begin{enumerate}
\item $\gl+\rho$ and $-w\rho$ lie in the closure of the same Weyl chamber, namely, $C=\{y\in\fh^*:\IP{y}{w\ga}<0, \text{ for }\ga\in\gD^+\}$.
\item There is a translation functor $T$ such that $T(L_w)=L(\gl)$.
\item $\AV(L(\gl))=\AV(L_w)$.
\end{enumerate}
\end{prop}

\begin{proof}
(1) As $L(\gl)$ and $L_w$ are highest weight Harish-Chandra modules $\gl+\rho$ and $-w\rho$ are $\gD^+(\fk)$-dominant.  Since $Y_\gl=Y_{-w\rho-\rho}$, $\gl+\rho$ and $-w\rho$ are nonnegative on the same roots in $\gD(\fp)$.

(2) The translation functor $T$ is  tensoring by the finite-dimensional $\fg$-representation of extreme weight $\gl +\rho +w\rho$, then projecting to infinitesimal character $\gl+\rho$.  When $\gl +\rho$ is nonsingular the statement is easy.  In the singular case $T(L_w)$ is potentially $0$.  However, our choice of $w$ guarantees $T(L_w)=L(\gl)$.  This follows from  \cite[Thm. 7.9]{Humphreys08} as follows.  One needs $\gl +\rho\in\widehat{C}$, the upper closure of the chamber $C$.  Part (1) says that $\gl +\rho \in\bar{C}$, so suppose that $\IP{\gl +\rho}{\ga}=0$.  For $\gl +\rho \in \widehat{C}$ we need $w\ga>0$.  This is (2) of Proposition \ref{prop:correct-w}.

(3) The translation functor $T$ goes from infinitesimal character $\rho$ to infinitesimal character $\gl+\rho$.  There is also a translation functor $T'$ going the other way.  Since $T(L_w)=L(\gl)$, we have  $\AV(L(\gl))\subseteq\AV(L_w)$.  For the opposite inclusion, we use the adjointness property of translation functors \cite[\S 7.2]{Humphreys08}:
\begin{equation*}
\Hom_\fg(X,T'(Y))\simeq\Hom_\fg(T(X),Y).
\end{equation*}
With $X=L_w$ and $Y=L(\gl)$, we get $L_w\subseteq T'(L(\gl))$,  so $\AV(L_w)\subset\AV(L(\gl))$.
\end{proof}

To complete the reduction to infinitesimal character $\rho$, observe that we have $\AV(L_w)=\AV(L(\gl))$ and $\width(Y_\gl)=\width(Y_{-w\rho-\rho})$.

\subsection{Simply laced case}\label{sec:sl}  Assume $\gD$ is simply laced.  We assume the infinitesimal character is integral, otherwise the highest weight module is irreducible and the associated variety is $\bar{\cO_r}=\fp^+$.  Theorem \ref{thm:main} follows from the following Proposition.

\smallskip

\begin{prop}\label{prop:sl} Suppose $\gD$ is simply laced and let $w\in \cW$.
\begin{enumerate}
\item For $k=0,1,\dots,r$, the set $\{L_w\mid  w\in\cW \text{ and }\AV(L_w)=\bar{\cO_k}\}$ is a cell and the corresponding cell representation is irreducible. 
\item If $\width(\gD(\fn\cap \Ad(w)\fn))=k,$ then $\AV(L_w)=\bar{\cO_k}$.  
\item The cardinality of a maximal set of strongly orthogonal roots in $\gD(\fn\cap \Ad(w)\fn)$ is the width.
\item  $\AV(L_w)=\nu(w).$
\end{enumerate}
\end{prop}
 
\begin{rem}  Each of these statements is often false in the non-simply laced cases.
\end{rem}

We shall use the following lemma without proof.   

\begin{lemma} Suppose $\gD$ is simply laced. For each $k=0,1,\dots,r$ there exists a highest weight Harish-Chandra  module $L_w$ having $\bar{\cO_k}$ as associated variety.
\end{lemma}

\begin{proof}
For classical $\fg_\R$, a case-by-case argument shows that  all $\bar{\cO_k}, k=0,1,\dots,r$, occur as associated variety of some $A_\fq(\gl)$ representation that is a highest weight module.  In this case, writing the Levi decomposition of $\fq$ as $\fl+\fu$, the associated variety is $K\cdot\fu\cap\fp$, which is easily computed.  For the exceptional cases we may argue as follows.  There exist $A_\fq(\gl)$ modules with associated varieties $\bar{\cO_k}$ for $k\neq 1$ (e.g., the holomorphic discrete series for $k=r$).  When $k=1$ this is  not the case; there exists no unitary representation of regular integral infinitesimal character having associated variety $\bar{\cO_1}$.  It is known (see  \cite{Wallach79}) that there exist unitary ladder representations.  These have associated varieties $\bar{\cO_1}$ (by computation of GK dimension directly from the $K$-types) and have integral infinitesimal character.  The infinitesimal character is however singular.  The translation principle produces a $(\fg,K)$-module  of infinitesimal character $\rho$ having the same associated variety (as in the proof of Proposition~\ref{prop:red}).  Alternatively, Corollary~4.7 of \cite{Vogan80} may be applied to find $w\in\cW$ so that $\AV(L_w)=\bar{\cO_1}$. Yet another alternative is to use the atlas software to view all cell decompositions (as $W$-representations) for cells having associated variety of annihilator $\bar{\cO_k^\C}$.
\end{proof}

The lemma implies there exists a cell $\fC_k$ of associated variety $\bar{\cO_k}$, for each $k=0,1,\dots,r$.  We know by \S\ref{ssec:AV}  that $\#\fC_k\geq \dim(\pi(\cO_k^\C,1))$.   Let $M_k:=\{w\in\cW\mid \width(\gD(\fn\cap \Ad(w)\fn))=k\}$ and let $m_k$ be the cardinality of $M_k$.  Our Tables \ref{tab:width} and \ref{tab:springer} show that 
\begin{equation}\label{eqn:spr-width}
 m_k=\dim(\pi(\cO_k^\C,1)),\quad k=0,1,\dots,r.
\end{equation} 

The following example illustrates the argument of our proof of the proposition.

\begin{example}\label{ex:e6} Consider  $\fe_{6(-14)}$.  The cardinality of $\cW$ is 27 and $r=2$.  So there are three $K$-orbits $\cO_0,\cO_1$ and $\cO_2$ in $\fp^+$.  Here is the data for $\fe_6$ (using the notation of \cite{Carter85} where $\phi_{a,b}$ has dimension $a$ and degree $b$).

\begin{table}[H] 
\centering
\renewcommand{\arraystretch}{1.5}
\setlength\tabcolsep{10pt}
\begin{tabular}{cccc}
\toprule
$i$  &  $\dim(\cO_i)$    & $\pi(\cO_i^\C,1)$  &  $\cO_i^\C$  \\
\midrule    
$0$  & $0$    & $\phi_{1,36}$  & $0$   \\
$1$  & $11$ &   $\phi_{6,25}$  &  $A_1$  \\
$2$  & $16$ &   $\phi_{20,20}$  &  $2A_1$\\
\bottomrule
\end{tabular}
\end{table}
There are three cells $\fC_0=\{\C\}, \fC_1$ and $\fC_2$.  We have $\#\fC_0=1$ and by the table  $\#\fC_1\geq 6$ and $\#\fC_2\geq 20$, but $\#\cW=27$, so equality holds and we conclude (a) there are no other cells and (b) each cell representation is irreducible.  In Table \ref{tab:width} we have $m_0=1, m_1=6$ and $m_2=20$.  By (\ref{eqn:key}), the only candidate for $\fC_0$ is $L_w$ with $\width(w)=0$ (so, of course, $L_w=\C$ and $w$ is the long element of $W$).  The only candidates for $\fC_1$ correspond to width $0$ or $1$ (by (\ref{eqn:key})), but width $0$ is already accounted for.  Therefore $\fC_1=\{L_w\mid \width(\gD(\fn\cap \Ad(w)\fn))=1\}$.  The remaining $20$ elements of $\cW$ have $\width(\gD(\fn\cap \Ad(w)\fn))=2$ and  must give us all of $\fC_2$.  
\end{example}

\begin{proof}[Proof of the proposition] Let $\fC_k$ be a cell with associated variety $\bar{\cO_k}$ for each $k$.  We know that $m_k=\dim(\pi(\cO_k^\C,1))$, by our tables, and $m_k\leq\#\fC_k$.  Since $\sum m_k=\#\cW=\sum \#\fC_k$ (obviously), we conclude
\begin{equation*}
m_k=\#\fC_k.
\end{equation*}
By (\ref{eqn:spr-width}), part (1) is proved.  Now (\ref{eqn:key}) implies
\begin{equation}\label{eqn:11}
\{L_w\mid \width(\gD(\fn\cap \Ad(w)\fn))\geq k\}\subseteq\bigcup_{j\geq k}\fC_j, \text{ all }k.
\end{equation}
Then
\begin{equation}\label{eqn:22}
\{L_w\mid \width(\gD(\fn\cap \Ad(w)\fn))=k\}=\fC_k, \text{ for all }k=0,1,\dots,r,
\end{equation}
by the following induction argument.  Use downward induction on $k$.   Take $k=r$ in (\ref{eqn:11}) to get $\{L_w\mid \width(\gD(\fn\cap \Ad(w)\fn))=r\}\subseteq\fC_r$, but the two sets have the same cardinality, so are equal.  For the inductive hypothesis assume $\{L_w\mid \width(\gD(\fn\cap \Ad(w)\fn))=j\}=\fC_j$ holds for $j\geq k+1$ and suppose $\width(\gD(\fn\cap \Ad(w)\fn))=k$.  Then $L_w\in\fC_j$ for some $j\geq k$.  By induction, $L_w\notin \fC_j$ for any $j\geq k+1$, so $L_w\in\fC_k$.  Again we get an inclusion $\{L_w\mid \width(\gD(\fn\cap \Ad(w)\fn))=k\}\subseteq\fC_k$ for which both sides have the same cardinality, so equality holds and (2) is proved.

For part (3), let
\begin{align*}
S_k:=\{w\in&\cW\mid \gD(\fn\cap \Ad(w)\fn)\text{ contains a maximal set of}\\ &\text{ strongly orthogonal roots having cardinality }k\}
\end{align*}
and let 
$$
M_k:=\{w\in\cW\mid \width(\gD(\fn\cap \Ad(w)\fn))=k\}.
$$
Then $M_k\subseteq \cup_{j\geq k}S_j$.  An argument very similar to the induction above  shows that $M_k=S_k$, for all $k$.

If $\width(\gD(\fn\cap \Ad(w)\fn))=k$, then, by part (2), Lemma \ref{lem:width-so} and (\ref{eqn:avBB}) we have
\begin{equation*}
\AV(L_w)=\bar{\cO_k}\subseteq\nu(w)\subset\AV(L_w),
\end{equation*}
and part (4) follows.
\end{proof}

\subsection{Non-simply laced cases: integral infinitesimal character}  The combinatorics in the simply vs. non-simply laced cases has a difference that is worth pointing out.  As an example, consider $\fg$ of type $B_3$.  The Hasse diagram for $\gD(\fp^+)=\{\ep_1\}\cup\{\ep_1\pm\ep_2,\ep_1\pm\ep_3\}$ a linear chain. 
There are five lower order ideals of width $1$.  Two contain a pair of strongly orthogonal long roots and three contain just one; the width is not the number of strongly orthogonal roots.  So (3) of Proposition \ref{prop:sl} fails.  It turns out that for all five $\AV(L_w)=K\cdot(X_{\ep_1-\ep_2}+X_{\ep_1+\ep_2})=\bar{\cO_2}$.  So part (2) fails in type $B_3$ for some $w\in\cW$. We shall see that (1) also fails.  Similarly, these statements fail for types $B_n$ and $C_n$, $n\geq 2$.

We prove the following for $\fg$ of type $B_n$  and $C_n$, $n\geq 2$.  Theorem \ref{thm:main} will follow (for all integral infinitesimal characters). 

\begin{claim}
Let $w\in \cW$ and $\width(\gD(\fn\cap \Ad(w)\fn))=k$.  Then the following hold.
\begin{enumerate}
\item If $\fg$ is of type $B_n$, then $\AV(L_w)=\bar{\cO_{2k}}$.
\item If $\fg$ is of type $C_n$, then 
\begin{equation*}
\AV(L_w)=\begin{cases} \bar{\cO_{2k}}, &2k\neq n+1  \\  \bar{\cO_n}, & 2k=n+1.
\end{cases}
\end{equation*}
\end{enumerate}    
\end{claim}

\begin{proof}[Proof of {\rm (1)}]
In this case $\#\cW=2n$ and $r=2$.  Table \ref{tab:springer} says that $\pi(\cO_0^\C,1)$  and $\pi(\cO_2^\C,1)$ are special and $\pi(\cO_1^\C,1)$ is not special.  Therefore, only $\bar{\cO_0}$ and $\bar{\cO_2}$ can occur as associated varieties.  We know that $\fC_0=\{\C\}$ , so the remaining $2n-1$ $L_w$'s have associated variety   $\bar{\cO_2}$.  For each of these $w$'s the width is $1$ (by Table \ref{tab:width}).  This proves (1) of the claim.  We also conclude there is just one cell of associated variety $\bar{\cO_2}$, since such a cell has size at least $\dim(\pi(\cO_k^\C,1))=n$.
\end{proof}

\begin{proof}[Proof of {\rm (2)}] This case is more involved; it requires more than the counting arguments.   Here $\#\cW=2^n$ and $r=n$.  The $\pi(\cO_k^\C,1)$ are special when either $k$ is even or $k=n$.  We begin with a lemma about the width of lower order ideals in $\gD(\fp^+)$ for type $C_n$. 

Consider the usual realization of $\fg=\f{sp}(2n,\C)$.  Then $\fp^+$ is identified with the symmetric $n\times n$ matrices $\sym(n,\C)$.  From the description of the $K$-orbits in $\fp^+$ (given  in (\ref{eqn:K-orbits})) it follows that 
\begin{align}\notag
\cO_j=\{x\in\sym(n,\C)\mid \rank(x)=j\},
\intertext{therefore}
\bar{\cO_j}=\{x\in\sym(n,\C)\mid \rank(x)\leq j\}.\label{eqn:rank}
\end{align}
\end{proof}

\begin{lemma}\label{lem:rank}  For type $C_n$, if $\width(\gD(\fn\cap \Ad(w)\fn))=k,$ then $\nu(w)= \bar{\cO_{2k}}$ or $\bar{\cO_{2k-1}}$, for $k=0,1,\dots,\lfloor\frac{n}{2}\rfloor$.  
\end{lemma}
\begin{proof}
Suppose $\width(\gD(\fn\cap \Ad(w)\fn))=k$, then $\gD(\fn\cap \Ad(w)\fn)$ contains $2\ep_{n},2\ep_{n-1},\dots,$ $2\ep_{n-k+1}$ and does not contain any roots $\ep_i+\ep_j$ with $i\leq k\leq j$.  Therefore, $\fn\cap \Ad(w)\fn$ is contained in the space of matrices  
\begin{equation*}
\left[\begin{array}{c|c}
0 & B \\ \hline B^t & D\end{array}\right]\subseteq\sym(n,\C), \text{ with }D\in \sym(k,\C).
\end{equation*}
These matrices all have rank at most $2k$, so $\nu(w)\subseteq \bar{\cO_{2k}}$.

In \S\ref{sec:width-sp} we gave the form of a maximal chain when the width is $k$.  Therefore, $\fn\cap\Ad(w)\fn$ contains either
\begin{equation*}
  \sum_{s=1}^{k-1}  X_{\ep_{n-k+1-s}+\ep_{n-k+1+s}}+X_{2\ep_{n-k+1}}\quad\text{ or }  \quad
  \sum_{s=1}^k X_{\ep_{n-k+1-s}+\ep_{n-k+s}},
\end{equation*}
matrices of rank either $2k-1$ or $2k$.  So $\nu(w)\supseteq\bar{\cO_{2k-1}}$.

Therefore $\bar{\cO_{2k-1}}\subseteq\nu(w)\subseteq \bar{\cO_{2k}}$.
\end{proof}
We need several known facts contained in \cite{BZ17} about the cell structure for type $C_n$, which we collect here.
\begin{lemma} Let $\fg$ be of type $C_n$.
\begin{itemize}
\item[(a)] The highest weight Harish-Chandra  modules of infinitesimal character $\rho$ are partitioned into cells: 
\begin{align*}
 &\fC_{2k}, \text{ having associated variety } \bar{\cO_{2k}},\ k=0,1,\dots,\lfloor n/2\rfloor \text{ and }  \\
 &\fC_n,  \text{ having associated variety }
 \bar{\cO_{n}},\text{ when $n$ is odd.} 
 \end{align*}
\item[(b)]These cells have cardinalities
 \begin{align*}
&\#\fC_{2k}=\dim{\pi(\cO_{2k}^\C,1)}+\dim{\pi(\cO_{2k-1}^\C,1)} \text{ and } \#\fC_n=\dim(\pi(\cO_{\frac{n+1}{2}}^\C,1)).
\end{align*}
\end{itemize}
\end{lemma}

\begin{proof} For part (b), see \cite[Prop.~3.7]{BZ17} and the preceding discussion in \cite{BZ17}.

Now we prove the claim by showing 
\begin{equation*}
\{L_w:\width(\gD(\fn\cap \Ad(w)\fn))=k\}=\begin{cases}\fC_{2k},& k=0,1,\dots,\lfloor \frac{n}{2}\rfloor\\ \fC_n,&k=\frac{n+1}{2}.\end{cases}
\end{equation*}
Table \ref{tab:width} along with part (a) of the lemma imply
\begin{equation}
\begin{split}\label{eqn:numbers-cn}
 &m_k=\dim{\pi(\cO_{2k}^\C,1)}+\dim{\pi(\cO_{2k-1}^\C,1)}=\#\fC_{2k} \text{ and}  \\
 &m_{\frac{n+1}{2}}=\dim(\pi(\cO_{\frac{n+1}{2}}^\C,1))=\#\fC_n, \  n \text{ odd}.
\end{split}
\end{equation}

If $\AV(L_w)\subset\bar{\cO_{2k}}$, then $\nu(w)\subset\bar{\cO_{2k}}$.  Therefore, by Lemma \ref{lem:rank}, $\width(\gD(\fn\cap\Ad(w)\fn))\leq k$.  It follows that 
\begin{equation*}
\bigcup_{j\leq k}\fC_{2j} \subset \{L_w\mid \width(\gD(\fn\cap\Ad(w)\fn))\leq k\}=\bigcup_{j\leq k}\{L_w\mid \width(\gD(\fn\cap\Ad(w)\fn))=j\},
\end{equation*}
for $k=0,1,\dots,\lfloor\frac{n}{2}\rfloor$.  Equality holds by (\ref{eqn:numbers-cn}).
Part (2) of the claim now follows by induction on $k$ (for $k=0,1,\dots,\lfloor\frac{n}{2}\rfloor).$  The $k=0$ case is clear.  If 
\begin{equation*}
 \{L_w\mid \width(\gD(\fn\cap \Ad(w)\fn))=j\}=\fC_{2j}, \text{ for }j< k
\end{equation*}
then 
\begin{equation*}
 \{L_w\mid \width(\gD(\fn\cap \Ad(w)\fn))=k\}=\bigcup_{j\leq k}\fC_{2j}\smallsetminus 
 \bigcup_{j< k}\fC_{2j}=\fC_{2k}.
\end{equation*}

Finally, when $n$ is odd, for the remaining $L_w$'s we have $\width(\gD(\fn\cap \Ad(w)\fn))=\frac{n+1}{2}$ and $\AV(L_w)=\bar{\cO_n}$.
\end{proof}

\subsection{Half-integral case}\label{ssec:half}
Theorem \ref{thm:main} will  now be proved for highest weight Harish-Chandra modules having half-integral highest weight.    As mentioned in the introduction, when $\gl$ is half-integral then either $L(\gl)$ is irreducible (in which case the associated variety is $\bar{\cO_r}=\fp^+$) \emph{or} $\gD=\gD(\fg,\fh)$ is non-simply laced.  We assume for the remainder of this section that $\gD$ is non-simply laced and $\gl$ is half-integral.

The two cases to consider are $\f{sp}(2n,\R)$ and $\f{so}(2,2n-1)$.  The integral root systems $\gD_\gl$ and corresponding Lie algebras $\fg_\gl$ are described in \S\ref{ssec:half-2}.  Recall that $W_\gl$ denotes the Weyl group of $\gD_\gl$ and we have fixed the positive system $\gD_\gl^+=\gD^+\cap\gD_\gl$.  Let $\Pi_\gl$ be the simple roots for this positive system.  The Lie algebra $\fg_\gl$ has a real form of Hermitian type and we my choose corresponding group $G_{\gl,\R}$ to be either $SO^*(2n)$ or $SL(2,\R)$.  Then, since $\gl$ is integral for $\gD_\gl$, $L(\gl)$ is an  Harish-Chandra  module for $G_{\gl,\R}$.  Furthermore, letting $K_\gl$ be the complexification of a maximal compact subgroup, the associated varieties of highest weight Harish-Chandra modules are the closures of the $K_\gl$-orbits in $\fp_\gl^+$, which we denote by $\cO_{\gl,0},\cO_{\gl,1},\dots$.

In our proof of Theorem \ref{thm:main} we first determine the Gelfand--Kirillov dimension of $L(\gl)$. This is done by applying  Lusztig's formula for the Gelfand--Kirillov dimension given in terms of his $a$-function.  Since the possible associated varieties are the $\bar{\cO_k}$, $k=0,1,\dots,r$, which form an increasing sequence, the Gelfand--Kirillov dimension determines the associated variety.  The key fact is that the $a$-function depends only on the Coxeter system $(W,\Pi)$, therefore we are able to relate the Gelfand--Kirillov dimension of $L(\gl)$ to the Gelfand--Kirillov dimension of a highest weight Harish-Chandra module for $\fg_{\gl,\R}$, then  we may apply the integral case of  Theorem \ref{thm:main} proved in the preceding sections.

 For Lusztig's $a$-function see \cite{Lusztig85a}, \cite{Lusztig85b}, \cite{BV83} and \cite{BX}.  However, details about the $a$-function are not needed here.  Lusztig's formula for the Gelfand--Kirillov dimension may be stated as follows.  If $\gl+\rho$ is  regular and $w\in W_\gl$ is chosen so that $w^{-1}(\gl+\rho)$ is antidominant, then 
\begin{subequations}
\begin{equation*}
 \GKD(L(\gl))=\#\gD^+ - a_\gl(w), 
\end{equation*}
where $a_\gl$ is Lusztig's $a$-function for $(W_\gl,\Pi_\gl)$.  In particular
\begin{equation}\label{eqn:LF2}
 \GKD(L_w)=\#\gD^+ - a(w), 
\end{equation}
where $a$ is the $a$-function for $(W,\Pi)$.
 
When $\gl+\rho$ is singular, then choose $w$ to be the minimal length element of $W_\gl$ for which $w^{-1}(\gl+\rho)$ is antidominant, then 
\begin{equation}\label{eqn:LF3}
 \GKD(L(\gl))=\#\gD^+ - a_\gl(w).
\end{equation}
\end{subequations}
See \cite[Prop.~1.1]{BX} for this formula.
  
\begin{prop}\label{prop:gkdim}
 Suppose $L(\gl)$ is a highest weight Harish-Chandra module and $\gl$ is half-integral.  If $\width(Y_\gl)=m$, then
\begin{equation*}
 \GKD(L(\gl))=(\#(\gD(\fp^+))-\#\gD(\fp_\gl^+)) + \dim(\cO_{\gl,m}). 
\end{equation*}
\end{prop}
\begin{proof}  Choose $w\in W_\gl$ to be the element of minimal length such that $w^{-1}(\gl+\rho)$ is antidominant. Then, by (\ref{eqn:LF3}),
 \begin{equation*}
   \GKD(L(\gl))=\#(\gD^+)-a_\gl(w), w\in W_\gl.
 \end{equation*}
 
Write $L_w'$ for the irreducible highest weight module for $\fg_\gl$ of highest weight $-w\rho_\gl-\rho_\gl$, $w\in W_\gl$.  Then formula (\ref{eqn:LF2}) (applied to $\fg_\gl$) gives
\begin{equation*}
 \GKD(L_w')=\#(\gD_\gl^+)-a_\gl(w).
\end{equation*}
 
Therefore, 
\begin{equation}\label{eqn:LF}
 \begin{split}
  \GKD(L(\gl))&=(\#(\gD^+)-\#(\gD_\gl^+))+\GKD(L_w')  \\
    &=(\#(\gD(\fp^+))-\#(\gD(\fp_\gl^+)))+\GKD(L_w'), \text{as }\gD(\fk)\subset\gD_\gl. 
 \end{split}
\end{equation}

Our choice of $w\in W_\gl$ guarantees that $Y_\gl=Y_{-w\rho_\gl-\rho_\gl}$ (by Proposition \ref{prop:correct-w} applied to $\gD_\gl$).  
Therefore, we have that $\width(Y_{-w\rho_\gl-\rho_\gl})=\width(Y_\gl)=m$.  By the integral case we know that $\AV(L_w')=\bar{\cO_{\gl,m}}$ and the proposition follows from (\ref{eqn:LF}). 
\end{proof}

Theorem \ref{thm:main} now follows easily for the half-integral case by comparing dimensions of  orbits using Table \ref{tab:springer}.

Assume $\width(Y_\gl)=m$.

\noindent (a) $\fg_\R=\f{sp}(2n,\R), n\geq 2$.  The real rank is $r=n$.
\begin{align*}
  \GKD(L(\gl))&=(n(n+1)-n(n-1))+\dim(\cO_{\gl,m})  \\
     &=2n+m(2n-2m-1), \text{ by Table \ref{tab:springer}}, \\
     &=(n-m)(2m+1)  \\
     &=\begin{cases} \dim(\cO_{2m+1}), &m<\frac{n}{2}  \\ 
        \dim(\cO_n), &m=\frac{n}{2}.
       \end{cases}
\end{align*}

\noindent (b) $\fg_\R=\f{so}(2,2n-1), n\geq 2$.  The real rank is $r=2$.
\begin{align*}
 \GKD(L(\gl))&=((2n-1)-1)+\dim(\cO_{\gl,m'})  \\
     &=2n-2+m  \\
     &=\begin{cases} \dim(\cO_1), &m=0 \\
                     \dim(\cO_2), &m=1 (=\frac{r}{2}).
       \end{cases}
\end{align*}

Therefore,
\begin{equation*}
 \AV(L(\gl))=\begin{cases} \cO_{2m+1}, & m<\frac{r}{2} \\
                          \cO_{r}, &m=\frac{r}{2}
             \end{cases}
\end{equation*}
in both cases.


\section{Another approach}\label{sec:alt-proof}


In this section, we will give another proof for Theorem \ref{thm:main}.
In \cite{BX} and  \cite{BX2}, Bai--Xie and Bai--Xiao--Xie have found an algorithm to compute the  Gelfand--Kirillov dimensions of highest weight modules of classical Lie algebras. We recall their results and will give another proof for our Theorem \ref{thm:main}.

For  a totally ordered set $ \Gamma $, we  denote by $ \mathrm{Seq}_n (\Gamma)$ the set of sequences $ x=(x_1,x_2,\cdots, x_n) $   of length $ n $ with $ x_i\in\Gamma $.
We say $q=(q_1, \cdots, q_N)$ is the \textit{dual partition} of a partition $p=(p_1, \cdots, p_N)$ and write $q=p^t$ if $q_i$ is the length of $i$-th column of the Young diagram $p$. 
Let $p(x)$ be the shape of the Young tableau $Y(x)$ obtained by applying Robinson--Schensted algorithm to $x\in \mathrm{Seq}_n (\Gamma)$. For convenience, we set $q(x)=p(x)^t$.

For a Young diagram $p$, use $ (k,l) $ to denote the box in the $ k $-th row and the $ l $-th column.
We say the box $ (k,l) $ is \textit{even} (resp. \textit{odd}) if $ k+l $ is even (resp. odd). Let $ p_i \ev$ (resp. $ p_i\od $) be the number of even (resp. odd) boxes in the $ i $-th row of the Young diagram $ p $. 
One can easily check that \begin{equation}\label{eq:ev-od}
p_i\ev=\begin{cases}
\left\lceil \frac{p_i}{2} \right\rceil&\text{ if } i \text{ is odd},\\
\left\lfloor \frac{p_i}{2} \right\rfloor&\text{ if } i \text{ is even},
\end{cases}
\quad p_i\od=\begin{cases}
\left\lfloor \frac{p_i}{2} \right\rfloor&\text{ if } i \text{ is odd},\\
\left\lceil \frac{p_i}{2} \right\rceil&\text{ if } i \text{ is even}.
\end{cases}
\end{equation}
Here for $ a\in \mathbb{R} $, $ \lfloor a \rfloor $ is the largest integer $ n $ such that $ n\leq a $, and $ \lceil a \rceil$ is the smallest integer $n$ such that $ n\geq a $. For convenience, we set
\begin{equation*}
p\ev=(p_1\ev,p_2\ev,\cdots)\quad\mbox{and}\quad p\od=(p_1\od,p_2\od,\cdots).
\end{equation*}

For $ x=(x_1,x_2,\cdots,x_n)\in \mathrm{Seq}_n (\Gamma) $, set
\begin{equation*}
\begin{aligned}
{x}^-=&(x_1,x_2,\cdots,x_{n-1}, x_n,-x_n,-x_{n-1},\cdots,-x_2,-x_1).
\end{aligned}
\end{equation*}

\begin{prop}[{\cite[Thm. 6.4]{BX}}]\label{thm:associatedA}
	Let $L(\lambda)$ be a highest weight Harish-Chandra module of $G_{\mathbb{R}}$ with $\mathfrak{g}_{\mathbb{R}}=\mathfrak{su}(p,q)$  and $\lambda+\rho=(t_1,\cdots,t_n)\in\mathfrak{h}^*$. Set $q=q(\lambda)=(q_1, \cdots, q_{n})$. Then  $AV(L(\lambda))=\overline{\mathcal{O}_{k'(\lambda)}}$ with
	\[
	k'(\lambda)=\begin{cases}
	q_2 & \text{ if } \lambda\  \text{is integral},\\
	\min\{p, q\} & \text{ otherwise}.
	\end{cases}
	\]
\end{prop}

\begin{prop}[{\cite[Thm. 6.2]{BX2}}]\label{thm:associated}
	Let $L(\lambda)$ be a highest weight Harish-Chandra module of $G_\R$ (classical Hermitian type) with $  \lambda+\rho=(t_1,\cdots,t_n)\in\mathfrak{h}^*$. Set $q=q(\lambda^-)=(q_1, \cdots, q_{2n})$. Then $AV(L(\lambda))=\overline{\mathcal{O}_{k'(\lambda)}}$ with $ k'(\lambda) $ given as follows.

	\begin{itemize}
		\item [(1)] $ \fg_\R=\sp(2n,\mathbb{R}) $ with $ n\geq 2 $. Then 
		\[
		k'(\lambda)=\begin{cases}
		2q_2\od & \text{ if } \lambda_1\in\mathbb{Z},\\
		2q_2\ev+1 & \text{ if } \lambda_1\in\frac12+\mathbb{Z},\\
		n & \text{ otherwise}.
		\end{cases}
		\]
		Note that $q_2\od=[\frac{q_2+1}{2}]$ and $q_2\ev=[\frac{q_2}{2}]$.
		\item [(2)] $\fg_\R=\so^*(2n) $ with $ n\geq 4 $. Then 
		\[
		k'(\lambda)=\begin{cases}
		q_2\ev & \text{ if } \lambda_1\in\frac12\mathbb{Z},\\
		\left\lfloor \frac{n}{2} \right \rfloor & \text{ otherwise}.
		\end{cases}
		\]
		\item [(3)] $\fg_\R=\so(2,2n-1) $ with $ n\geq 3 $.   Then\[
		k'(\lambda)=\begin{cases}
		0& \text{ if } \lambda_1-\lambda_2\in\mathbb{Z}, \lambda_1>\lambda_2,\\
		1 & \text{ if } \lambda_1-\lambda_2\in\frac12+\mathbb{Z}, \lambda_1>0,\\
		2 & \text{ otherwise}.
		\end{cases}
		\]
		\item [(4)] $\fg_\R=\so(2,2n-2)$ with $ n\geq 4 $.   Then\[
		k'(\lambda)=\begin{cases}
		0& \text{ if } \lambda_1-\lambda_2\in\mathbb{Z}, \lambda_1>\lambda_2,\\
		1 &  \text{ if } \lambda_1-\lambda_2\in \mathbb{Z}, -|\lambda_n|<  \lambda_1\leq \lambda_2\\
		2 & \text{ otherwise}.
		\end{cases}
		\]
	\end{itemize}
\end{prop}

Let $\Delta$ be the root system  of $E_6$ or $E_7$. As usual $\Delta$ can be realized as a subset of $\mathbb{R}^8$. The simple roots are 
\begin{align*}
\alpha_1&=\frac{1}{2}(\ep_1-\ep_2-\ep_3-\ep_4-\ep_5-\ep_6-\ep_7+\ep_8) ,\\
\alpha_2&=\ep_1+\ep_2,\\
\alpha_k&=\ep_{k-1}-\ep_{k-2} \text{ with }3\leq k\leq n,
\end{align*}
where $n=6, 7$. 
Here we use notations in \cite{EHW:83}. Let $\Pi$ denote the set of simple roots in $\Delta^+$. Exactly one root in $\Pi$ is in $\Delta(\fp^+)$.
In particular, for $\mathfrak{e}_{6(-14)}$, this simple root is $\alpha_1$; for $\mathfrak{e}_{7(-25)}$, this simple root is  $\alpha_7$. 


For an integral weight $ \lambda\in \mathfrak{h}^*$, recall that $ \Delta_\lambda^+=\{\alpha\in\Delta^+\mid\IP{\lambda+\rho}{\alpha^\vee}>0\} $.
\begin{prop}[{\cite[Thm. 7.1]{BX}}]\label{67}
	Let $L(\lambda)$ be a highest weight Harish-Chandra module of  $G_{\mathbb{R}}$. If $\lambda\in\mathfrak{h}^*$ is not integral, then $AV(L(\lambda))=\overline{\mathcal{O}_{k'(\lambda)}}$ with $k'(\lambda)=2$ (for $\mathfrak{g}_{\mathbb{R}}=\mathfrak{e}_{6(-14)}$) or $3$ (for $\mathfrak{g}_{\mathbb{R}}=\mathfrak{e}_{7(-25)}$). In the case when $\lambda$ is integral, $ k'(\lambda) $ is given as follows.
	\begin{itemize}
		\item [(1)] For $\mathfrak{g}_{\mathbb{R}}=\mathfrak{e}_{6(-14)}$, we have
		\[
		k'(\lambda)=\begin{cases}
		0& \text{ if } \Delta_\lambda^+\cap A_1\neq\emptyset,\\
		1& \text{ if } \Delta_\lambda^+\cap A_2\neq\emptyset \text{ and } \Delta_\lambda^+\cap A_1=\emptyset ,\\
		2& \text{ if } \Delta_\lambda^+\cap A_2=\emptyset,
		\end{cases}
		\]
		where $A_2=\{\frac{1}{2}(\pm(\ep_1+\ep_2+\ep_3-\ep_4)-\ep_5-\ep_6-\ep_7+\ep_8)\}$ and $A_1=\{\alpha_1\}$.
		
		\item [(2)] For $\mathfrak{g}_{\mathbb{R}}=\mathfrak{e}_{7(-25)}$, we have
		\[
		k'(\lambda)=\begin{cases}
		0& \text{ if } \Delta_\lambda^+\cap A_1\neq\emptyset,\\
		1& \text{ if } \Delta_\lambda^+\cap A_2\neq\emptyset \text{ and } \Delta_\lambda^+\cap A_1=\emptyset ,\\
		2& \text{ if } \Delta_\lambda^+\cap A_3\neq\emptyset \text{ and } \Delta_\lambda^+\cap A_2=\emptyset ,\\
		3& \text{ if } \Delta_\lambda^+\cap A_3=\emptyset,
		\end{cases}
		\]
		where $A_3=\{\frac{1}{2}(\pm(\ep_1-\ep_2-\ep_3+\ep_4)-\ep_5+\ep_6-\ep_7+\ep_8), \ep_5+\ep_6\}$, $A_2=\{\pm\ep_1+\ep_6\}$ and $A_1=\{-\ep_5+\ep_6\}$.
	\end{itemize}
	
\end{prop}

\begin{lemma}\label{Minimax}
	Suppose $\mu=(t_1,...,t_n, s_1,...,s_m)$ is a $(n,m)$-dominant sequence with $t_n\leq s_1$.   Then applying the Schensted insertion algorithm to $\mu$. We can get a Young tableau  $Y(\mu)$  with two columns. And $c_2(Y(\mu))$  is precisely the largest integer $k$ for which we have
	$$t_{n-k+1}\leq s_{1},...,t_{n-1}\leq s_{k-1},t_{n}\leq s_{k}.$$	
	The integer $k $ is also equal to the maximum integer $k_1$ for which  there exists a sequence of indices
	$$1\leq i_1<i_2<...<i_{k_1}\leq n,  1\leq l_1<l_2<...<l_{k_1}\leq m,$$
	such that
	\begin{equation}  t_{i_1}\leq s_{l_1},...,t_{i_{k_1}}\leq s_{l_{k_1}}.
	\end{equation}
\end{lemma}

\begin{proof}In  \cite{BX},  a  white-black pair model was  constructed to compute $c_2(Y(\mu))$. From this model and  \cite[Prop. 5.3 and Thm. 5.6]{BX},
	we know $k(\mu):=c_2(Y(\mu))$ is equal to the maximum integer $k$ for which we can divide the sequence $(s_1,...,s_m)$ into several parts such that $t_n\leq s_{j_k}$ and $t_n> s_{j_k+1}$,...,$t_{n-k+1}\leq s_{j_{1}}$ and $t_{n-k+1}> s_{j_1+1}$, with $1\leq j_1<j_2<...<j_k\leq m$. Note that $j_k$ is the first index sucht $t_n\leq s_j$.
	
	We denote a  maximum integer $k_1$ for which  there exists a sequence of indices
	$$1\leq i_1<i_2<...<i_{k_1}\leq n,  1\leq l_1<l_2<...<l_{k_1}\leq m,$$
	such that
	\begin{equation}\label{mn}  t_{i_1}\leq s_{l_1},...,t_{i_{k_1}}\leq s_{l_{k_1}}.
	\end{equation}
	
	So we have $k\leq k_1$. We suppose $k<k_1$.
	
	Since $t_i-t_j \in \mathbb{Z}_{> 0}$ for $1\leq i<j\leq n$,  we can make a new choice for these indices such that $$i_1=n-k_{1}+1, i_2=n-k_1,...,i_{k_1}=n.$$
	Then we still have the condition (\ref{mn}): $$ t_{n-k_{1}+1}\leq s_{l_1},...,t_{n}\leq s_{l_{k_1}}.$$
	From our white-black pair model, we can choose $l_{k_1}=j_k$ since  $t_n\leq s_{j_k}$ and $t_n> s_{j_k+1}$ (i.e., $j_k$ is the first index such that $t_n\leq s_j$). Similarly we can choose
	$l_{k_1-1}=j_{k-1}$,...,$l_{k_1-k+1}=j_1$. Then we choose $l_{k_1-k}=j'_1$ such that $t_{n-k}\leq s_{j'_1}=s_{l_{k_1-k}}$ and $t_{n-k}>s_{j'_1+1}$. We continue this process until we choose $l_1=j'_{k_1-k}$ such that
	$t_{n-k_1+1}\leq s_{j'_{k_1-k}}=s_{l_1}$ and $t_{n-k_1+1}> s_{j'_{k_1-k+1}}$.
	Then from the definition of $k$, we must have $k_1\leq k$, which is a contradiction. So we have $k=k_1$.

	Since $s_i-s_j \in \mathbb{Z}_{> 0}$ for $1\leq i<j\leq m$, we make a new choice such that  $$ j_1=1,...,j_k=k.$$
	Then we get $$t_{n-k+1}\leq s_{1},...,t_{n-1}\leq s_{k-1},t_{n}\leq s_{k}.$$	
	
	Since this process is invertible, we can see that $k$ is precisely the largest number satisfying the above condition.

\end{proof}

\begin{cor}\label{pair}
	Suppose $\lambda=(t_1,...,t_n, -t_n,...,-t_1)$ is a $(n,n)$-dominant sequence with $t_n\leq -t_n$.   Then applying the Schensted insertion algorithm to $\lambda$. We can get a Young tableau  $Y_{\lambda}$  with two columns. And $c_2(Y_{\lambda})$  is precisely the largest integer $t$ for which we have
	$$t_{n-t+1}\leq -t_{n},t_{n-t+2}\leq -t_{n-1}...,t_{n-1}\leq -t_{n-t+2},t_{n}\leq -t_{n-t+1}.$$
	
\end{cor}

Now we will give our proof for Theorem \ref{thm:main}  by the algorithms in \cites{BX,BX2}. In the following we denote $m(\lambda):=\operatorname{width}(Y_\lambda)$. We will show that the values of $k'(\lambda)$ in Proposition \ref{thm:associatedA}, Proposition \ref{thm:associated} and Proposition \ref{67} will equal to the values of $k(\lambda)$ given  in Theorem \ref{thm:main}.

\subsection{$\mathfrak{g}_{\mathbb{R}}=\mathfrak{su}(p,q)$}
We use the notations from \cite{EHW:83}. Suppose $L(\lambda)$ is a  highest weight Harish-Chandra ($\mathfrak{g}$,$K$)-module with highest weight $\lambda$. We denote $\lambda+\rho=(t_1,...,t_n)$. Suppose $\lambda$ is integral. Then from \cite{EHW:83}, we have $t_i-t_j \in \mathbb{Z}_{> 0}$ for $1\leq i\leq p$, $p+1\leq j\leq n$. From \cite{BX}, we say that $(t_1,...,t_n)$ is $(p,q)$-dominant.

By using Schensted insertion  algorithm for $(t_1,...,t_n)$, we can get  a Young tableau $Y_{\lambda}$ which consists of
at most two columns. Now we suppose the number of entries in the second column of  $Y_{\lambda}$ is $c_2(Y_{\lambda})=q_2$, then $c_1(Y_{\lambda})=n-q_2$. From Proposition \ref{thm:associatedA}, we know  
$ AV(L(\lambda))= \overline{\mathcal{O}_{q_2}}. $

Then from Lemma \ref{Minimax}, we know $q_2$ is equal to the maximum integer $t$ for which there exists a sequence of indices
$$1\leq i_1<i_2<...<i_t\leq p, ~p+1\leq j_1<j_2<...<j_t\leq n,$$
such that
$$t_{i_1}\leq t_{j_1},...,t_{i_t}\leq t_{j_t}.$$

For $\mathfrak{g}_{\mathbb{R}}=\mathfrak{su}(p,q)$, we know $\Delta(\fp^+)=\{\eps_i-\eps_j| 1\leq i\leq p, ~p+1\leq j\leq n\}$.  So  $m=q_2$ is just our $m(\lambda)$ in Theorem \ref{thm:main}.

Thus  $AV(L(\lambda))=\overline{\mathcal{O}_{q_2}}=\overline{\mathcal{O}_{m}}=\overline{\mathcal{O}_{m(\lambda)}}=\overline{\mathcal{O}_{k(\lambda)}}$.

When $\lambda$ is non-integral, we will have $N(\lambda)=L(\lambda)$ by \cite[Lem. 3.17]{EHW:83}. So $AV(L(\lambda))=\overline{\mathcal{O}_{r}}$.

\subsection{$\mathfrak{g}_{\mathbb{R}}=\mathfrak{sp}(2n,\mathbb{R})$}
 Suppose $L(\lambda)$ is a  highest weight Harish-Chandra module  with highest weight $\lambda$. We denote $\lambda+\rho=(t_1,...,t_n)$. Suppose $\lambda$ is integral. Then from \cite{EHW:83}, we have $t_i-t_j \in \mathbb{Z}_{> 0}$ for $1\leq i<j\leq n$. From \cite{BX}, we say that $(t_1,...,t_n,-t_n,-t_{n-1},...,-t_2,-t_1)$ is $(n,n)$-dominant.

By using Schensted insertion  algorithm for $(t_1,...,t_n,-t_n,-t_{n-1},...,-t_2,-t_1)$, we can get  a Young tableau $Y(\lambda^-)$ which consists of
at most two columns. Now we suppose the number of entries in the second column of  $Y(\lambda^-)$ is $c_2(Y(\lambda^-))=q_2$, then $c_1(Y(\lambda^-))=2n-q_2$. From Proposition \ref{thm:associated}, we know  
$$ AV(L(\lambda))= \bar{\mathcal{O}_{2q_2\od}}=\left\{
\begin{array}{ll}
\bar{\mathcal{O}_{q_2+1}}, & \hbox{if $q_2$ is odd,} \\
\bar{\mathcal{O}_{q_2}}, & \hbox{if $q_2$ is even.}
\end{array}
\right. $$

Then from Lemma \ref{Minimax}, we know $q_2$ is equal to the maximum integer $t$ for which there exists a sequence of indices
$$1\leq i_1<i_2<...<i_t\leq n, n\geq j_1>j_2>...>j_t\geq 1,$$
such that
$$t_{i_1}\leq -t_{j_1},...,t_{i_t}\leq -t_{j_t}.$$
Since $t_i-t_j \in \mathbb{Z}_{> 0}$ for $1\leq i<j\leq n$, we can make a new choice for these indices such that $$i_1=j_t, i_2=j_{t-1},...,i_t=j_1. $$
Thus we have
$$t_{i_1}\leq -t_{i_t}, t_{i_2}\leq -t_{i_{t-1}}, ..., t_{i_t}\leq -t_{i_1}.$$

When $q_2=t=2t'+1$ is an odd number, we will have 
$$t_{i_1}+t_{i_t}\leq 0, t_{i_2}+t_{i_{t-1}}\leq 0, ..., t_{i_{t'}}+t_{i_{t'+2}}\leq 0, t_{i_{t'+1}}\leq 0.$$

From from Lemma \ref{Minimax}, we know $m=t'$ is the maximum integer  for which there exists a sequence of indices
$$1\leq i_1<i_2<...<i_{m} < l<j_{m}<...<j_1\leq n,$$
such that
$$t_{i_1}+t_{j_1}\leq 0,...,t_{i_{m}}+t_{j_{m}}\leq 0, t_l\leq 0.$$
For $\mathfrak{g}_{\mathbb{R}}=\mathfrak{sp}(2n,\mathbb{R})$, we know $\Delta(\fp^+)=\{\eps_i+\eps_j| 1\leq i< j\leq n\}\cup \{2\eps_l| 1\leq l \leq n\}$.  So  $m+1$ is  just our $m(\lambda)$ in Theorem \ref{thm:main}.

Thus $q_2+1=t+1=2(m+1)$  and $AV(L(\lambda))=\overline{\mathcal{O}_{q_2+1}}=\overline{\mathcal{O}_{2(m+1)}}=\overline{\mathcal{O}_{2m(\lambda)}}==\overline{\mathcal{O}_{k(\lambda)}}$.

When $q_2=t=2t'$ is an even number, we will have 
$$t_{i_1}+t_{i_t}\leq 0, t_{i_2}+t_{i_{t-1}}\leq 0, ..., t_{i_{t'}}+t_{i_{t'+1}}\leq 0.$$

From the definition of $t$, we know $m=t'$ is the maximum integer  for which there exists a sequence of indices
$$1\leq i_1<i_2<...<i_{m} < l<j_{m}<...<j_1\leq n,$$
such that
$$t_{i_1}+t_{j_1}\leq 0,...,t_{i_{m}}+t_{j_{m}}\leq 0, t_l\leq 0.$$
So  $m$ is  just our $m(\lambda)$ in Theorem \ref{thm:main}.

Thus $q_2=t=2m$  and $AV(L(\lambda))=\overline{\mathcal{O}_{q_2}}=\overline{\mathcal{O}_{2m}}=\overline{\mathcal{O}_{2m(\lambda)}}=\overline{\mathcal{O}_{k(\lambda)}}$.

Now we suppose $\lambda$ is half-integral. Then from \cite{EHW:83}, we have $t_i-t_j \in \mathbb{Z}_{> 0}$ for $1\leq i<j\leq n$ and $t_k-\frac{1}{2} \in \mathbb{Z}$ for $1\leq k\leq n$.
By using Schensted insertion  algorithm for $(t_1,...,t_n,-t_n,-t_{n-1},...,-t_2,-t_1)$, we can get  a Young tableau $Y(\lambda^-)$ which consists of
at most two columns. Now we suppose the number of entries in the second column of  $Y(\lambda^-)$ is $c_2(Y(\lambda^-))=q_2$, then $c_1(Y(\lambda^-))=2n-q_2$. 
From Proposition \ref{thm:associated}, we know  
$$ AV(L(\lambda))= \bar{\mathcal{O}_{2q_2\ev+1}}=\left\{
\begin{array}{ll}
\bar{\mathcal{O}_{q_2}}, & \hbox{if $q_2$ is odd,} \\
\overline{\mathcal{O}_{q_2+1}}, & \hbox{if $q_2$ is even.}
\end{array}
\right. $$

When $q_2=t=2t'+1$ is an odd number, similarly to the integral case, we will have 
$$t_{i_1}+t_{i_t}\leq 0, t_{i_2}+t_{i_{t-1}}\leq 0, ..., t_{i_{t'}}+t_{i_{t'+2}}\leq 0, t_{i_{t'+1}}\leq 0,$$
and $m=t'$ is the maximum integer  for which there exists a sequence of indices
$$1\leq i_1<i_2<...<i_{m} < l<j_{m}<...<j_1\leq n,$$
such that
$$t_{i_1}+t_{j_1}\leq 0,...,t_{i_{m}}+t_{j_{m}}\leq 0, t_l\leq 0.$$
For $\mathfrak{g}_{\mathbb{R}}=\mathfrak{sp}(2n,\mathbb{R})$, we know $\Delta(\fp^+)=\{\eps_i+\eps_j| 1\leq i< j\leq n\}\cup \{2\eps_l| 1\leq l \leq n\}$.  Since $t_l\notin \mathbb{Z}$,  $m$ is  just our $m(\lambda)$ in Theorem \ref{thm:main}.

Thus $q_2=t=2m+1$  and $AV(L(\lambda))=\overline{\mathcal{O}_{q_2}}=\overline{\mathcal{O}_{2m+1}}=\overline{\mathcal{O}_{2m(\lambda)+1}}=\overline{\mathcal{O}_{k(\lambda)}}$.

When $q_2=t=2t'$ is an even number, similarly to above case, we will have $m=m(\lambda)=t'$ and $q_2+1=2m+1$. Thus $AV(L(\lambda))=\overline{\mathcal{O}_{q_2+1}}=\overline{\mathcal{O}_{2m+1}}=\overline{\mathcal{O}_{2m(\lambda)+1}}=\overline{\mathcal{O}_{k(\lambda)}}$.

\subsection{$\mathfrak{g}_{\mathbb{R}}=\mathfrak{so}^*(2n)$}

Suppose $L(\lambda)$ is a  highest weight Harish-Chandra module  with highest weight $\lambda$. We denote $\lambda+\rho=(t_1,...,t_n)$. Then from \cite[Lem. 3.17]{EHW:83}, we may suppose $\lambda$ is integral (otherwise $N(\lambda)$ will be irreducible). Also we have   $t_1+t_2\in \mathbb{Z}$  and $t_i-t_j \in \mathbb{Z}_{> 0}$ for $1\leq i<j\leq n$. 
Then we have $$t_1>t_2>...>t_{n-1}>t_n, -t_n>-t_{n-1}>...>-t_1. $$ From \cite{BX}, we say that $(t_1,...,t_n,-t_n,-t_{n-1},...,-t_2,-t_1)$ is $(n,n)$-dominant.
When $t_n\geq 0$, $L(\lambda)$ will be finite-dimensional. We only consider infinite-dimensional modules. Thus we suppose $t_n<0$.

By using Schensted insertion  algorithm for $(t_1,\dots,t_n,-t_n,-t_{n-1},\ldots,-t_2,-t_1)$, we can get  a Young tableau $Y(\lambda^-)$ which consists of
at most two columns. Now we suppose the number of entries in the second column of  $Y(\lambda^-)$ is $c_2(Y(\lambda^-))=q_2$, then $c_1(Y(\lambda^-))=2n-q_2$. 
From Proposition \ref{thm:associated}, we know  
$$ AV(L(\lambda))= \bar{\mathcal{O}_{[\frac{q_2}{2}]}}=\left\{
\begin{array}{ll}
\bar{\mathcal{O}_{\frac{q_2-1}{2}}}, & \hbox{if $q_2$ is odd,} \\
\bar{\mathcal{O}_{\frac{q_2}{2}}}, & \hbox{if $q_2$ is even.}
\end{array}
\right. $$

Then similar to the case of $\sp(2n,\R)$,
we know $q_2$ is equal to the maximum integer $t$ for which there exists a sequence of indices
$$1\leq i_1<i_2<...<i_t\leq n,$$
such that
$$t_{i_1}\leq -t_{i_t}, t_{i_2}\leq -t_{i_{t-1}}, ..., t_{i_t}\leq -t_{i_1}.$$

When $q_2=t=2t'+1$ is an odd number, we will have 
$$t_{i_1}+t_{i_t}\leq 0, t_{i_2}+t_{i_{t-1}}\leq 0, ..., t_{i_{t'}}+t_{i_{t'+2}}\leq 0, t_{i_{t'+1}}\leq 0.$$

From the definition of $t$, we know $m=t'$ is the maximum integer  for which there exists a sequence of indices
$$1\leq i_1<i_2<...<i_{m} <j_{m}<...<j_1\leq n,$$
such that
$$t_{i_1}+t_{j_1}\leq 0,...,t_{i_{m}}+t_{j_{m}}\leq 0.$$

For $\mathfrak{g}_{\mathbb{R}}=\mathfrak{so}^*(2n)$, we know $\Delta(\fp^+)=\{\eps_i+\eps_j| 1\leq i< j\leq n\}$.  So  $m=t'=\frac{q_2-1}{2}$ is  just our $m(\lambda)$ in Theorem \ref{thm:main}.
Thus  $AV(L(\lambda))=\overline{\mathcal{O}_{\frac{q_{2}-1}{2}}}=\overline{\mathcal{O}_{m}}=\overline{\mathcal{O}_{m(\lambda)}}=\overline{\mathcal{O}_{k(\lambda)}}$.

When $q_2=t=2t'$ is an even number, similarly we will have 
$AV(L(\lambda))=\overline{\mathcal{O}_{\frac{q_2}{2}}}=\overline{\mathcal{O}_{m}}=\overline{\mathcal{O}_{m(\lambda)}}=\overline{\mathcal{O}_{k(\lambda)}}$.


\subsection{$\mathfrak{g}_{\mathbb{R}}=\mathfrak{so}(2,2n-1)$}
 Suppose $L(\lambda)$ is a  highest weight Harish-Chandra module with highest weight $\lambda$. We denote $\lambda+\rho=(t_1,...,t_n)$. Suppose $\lambda$ is integral. Then from \cite{EHW:83}, we have $t_n> 0$, $t_1-t_2\in \mathbb{Z}$, $2t_k\in  \mathbb{Z}$ for $1\leq k\leq n $ and $t_i-t_j \in \mathbb{Z}_{> 0}$ for $2\leq i<j\leq n$. Then we have $t_2>t_3>...>t_n>-t_n>...>-t_2 $.
 When $t_1>t_2$, $L(\lambda)$ will be finite-dimensional and $V(L(\lambda))=\overline{\mathcal{O}}_{0}$. We only consider infinite-dimensional modules. Thus we suppose $t_1\leq t_2$.
 
By using Schensted insertion  algorithm for $(t_1,...,t_n,-t_n,-t_{n-1},...,-t_2,-t_1)$, we can get  a Young tableau $Y(\lambda^-)$ which consists of
at most three columns. From the construction process, we can see that $Y(\lambda^-)$ will be a Young tableau consisting of two columns with $c_1(Y(\lambda^-))=2n-2$ and $c_2(Y(\lambda^-))=2$ when $t_2\geq t_1>-t_n$, and it will be a Young tableau consisting of three columns with $c_1(Y(\lambda^-))=2n-2$, $c_2(Y(\lambda^-))=1$ and $c_3(Y(\lambda^-))=1$ when $t_1\leq -t_n$.

From Proposition \ref{thm:associated}, we have  $AV(L(\lambda))=\overline{\mathcal{O}_{2}}$ for both cases.

For $\mathfrak{g}_{\mathbb{R}}=\mathfrak{so}(2,2n-1)$, we know $\Delta(\fp^+)=\{\eps_1\pm \eps_j| 2\leq i< j\leq n\}\cup \{\eps_l| 1\leq l \leq n\}$. When $t_2\geq t_1>-t_n$, we have $t_1-t_2\leq 0$.
So  $m=1$ is  just our $m(\lambda)$ in Theorem \ref{thm:main}  and $AV(L(\lambda))=\overline{\mathcal{O}_{2}}=\overline{\mathcal{O}_{2m(\lambda)}}=\overline{\mathcal{O}_{k(\lambda)}}$.

When $t_1\leq -t_n$, we have $t_1+t_n\leq 0$.
So  $m=1$ is  just our $m(\lambda)$ in Theorem \ref{thm:main} and $AV(L(\lambda))=\overline{\mathcal{O}_{2}}=\overline{\mathcal{O}_{2m(\lambda)}}=\overline{\mathcal{O}_{k(\lambda)}}$.

Now we suppose $\lambda$ is half-integral. Then from Enright--Howe--Wallach \cite{EHW:83}, we have $t_n> 0$, $t_1-t_2\in \frac{1}{2}+\mathbb{Z}$, $2t_k\in  \mathbb{Z}$ for $1\leq k\leq n $ and $t_i-t_j \in \mathbb{Z}_{> 0}$ for $2\leq i<j\leq n$. Then we have $t_2>t_3>...>t_n>-t_n>...>-t_2 $.

From Proposition \ref{thm:associated}, we have   $AV(L(\lambda))=\left\{
\begin{array}{ll}
\bar{\mathcal{O}_{1}}, & \hbox{if  $t_1>0$,} \\
\bar{\mathcal{O}_{2}}, & \hbox{if $t_1\leq 0$.}
\end{array}
\right.$

When $t_1>0$, we find that $m(\lambda)=0$. So  $AV(L(\lambda))=\overline{\mathcal{O}_{1}}=\overline{\mathcal{O}_{2m(\lambda)+1}}=\overline{\mathcal{O}_{k(\lambda)}}$.

When $t_1\leq 0$, we have $m(\lambda)=1$ and $2m(\lambda)+1=3>2$.
We still have $AV(L(\lambda))=\overline{\mathcal{O}_{2}}=\overline{\mathcal{O}_{k(\lambda)}}$.

\subsection{$\mathfrak{g}_{\mathbb{R}}=\mathfrak{so}(2,2n-2)$}

 Suppose $L(\lambda)$ is a  highest weight Harish-Chandra module  with highest weight $\lambda$. We denote $\lambda+\rho=(t_1,...,t_n)$. Then from \cite[Lem. 3.17]{EHW:83}, we may suppose $\lambda$ is integral (otherwise $N(\lambda)$ will be irreducible). Also we have  $t_1- t_2\in \mathbb{Z}$, $t_{n-1}-|t_n|>0$ and $t_i-t_j \in \mathbb{Z}_{> 0}$ for $2\leq i<j\leq n$. 
 Then we have $$t_2>t_3>...>t_{n-1}>|t_n|\geq 0\geq -|t_n|>-t_{n-1}>...>-t_2. $$
When $t_1>t_2$, $L(\lambda)$ will be finite-dimensional and $AV(L(\lambda))=\overline{\mathcal{O}_{0}}$. We only consider infinite-dimensional modules. Thus we suppose $t_1\leq t_2$.

By using Schensted insertion  algorithm for $(t_1,...,t_n,-t_n,-t_{n-1},...,-t_2,-t_1)$, we can get  a Young tableau $Y(\lambda^-)$ which consists of
at most four columns.

From Proposition \ref{thm:associated}, when $-|t_n|<  t_1\leq t_2$,  then $ q(\lambda^-)=(2n-2, 2)$ or $(2n-3,3)$, we get $AV(L(\lambda))=\overline{\mathcal{O}_{1}}$.

For $\mathfrak{g}_{\mathbb{R}}=\mathfrak{so}(2,2n-2)$, we know $\Delta(\fp^+)=\{\eps_1\pm \eps_j| 2\leq i< j\leq n\}$. For $t_n> 0$ and $t_2\geq t_1>-t_n$, we have $t_1-t_2\leq 0$ and $t_1+t_n>0$.
For $t_n\leq 0$ and $t_2\geq t_1> t_n$, we have $t_1-t_2\leq 0$ and $t_1-t_n>0$. So for both cases,  $m=1$ is  just our $m(\lambda)$ in Theorem \ref{thm:main} and $AV(L(\lambda))=\overline{\mathcal{O}_{1}}=\overline{\mathcal{O}_{m(\lambda)}}=\overline{\mathcal{O}_{k(\lambda)}}$.

 From Proposition \ref{thm:associated}, when $ t_1\leq -|t_n| $, then $ q(\lambda^-)=(2n-2, 1^2)$ or $(2n-3, 1^3)$, one has $AV(L(\lambda))=\overline{\mathcal{O}_{2}}$.

%

%

When $t_1\leq -t_n<0$, we have $t_1+t_n\leq 0$ and $t_1-t_n\leq 0$. When $t_1\leq t_n\leq 0$, we have $t_1-t_n\leq 0$ and $t_1+t_n\leq 0$.
So for both cases,  $m=2$ is  just our $m(\lambda)$ in Theorem \ref{thm:main} and $AV(L(\lambda))=\overline{\mathcal{O}_{2}}=\overline{\mathcal{O}_{m(\lambda)}}=\overline{\mathcal{O}_{k(\lambda)}}$.


\subsection{$\mathfrak{g}_{\mathbb{R}}=\mathfrak{e}_{6(-14)}$}

Suppose $L(\lambda)$ is a  highest weight Harish-Chandra module with highest weight $\lambda$. We denote $\lambda+\rho=(t_1,...,t_8)$. Then from \cite[Lem. 3.17]{EHW:83}, we may suppose $\lambda$ is integral (otherwise $N(\lambda)$ will be irreducible). Also we have $|t_1|< t_2<...<  t_5$, $t_i-t_j\in \mathbb{Z}$ and $2t_i\in \mathbb{Z}$ for all $1\leq i, j\leq 5$.

The Hasse diagram of $\Delta(\fp^+)$ for $\mathfrak{g}_{\mathbb{R}}=\mathfrak{e}_{6(-14)}$ is shown in the Appendix.


The highest root is $\theta=  \alpha_{1}+2\alpha_{2}+2\alpha_{3}+3\alpha_{4}+2\alpha_{5}+\alpha_{6}=\frac{1}{2}(1,1,1,1,1,-1,-1,1)$. 

We denote $$\beta_1=\alpha_1+\alpha_3+\alpha_4+\alpha_5=\frac{1}{2}(-1,-1,-1,1,-1,-1,-1,1)$$ and $$\beta_2=\alpha_1+\alpha_3+\alpha_4+\alpha_2=\frac{1}{2}(1,1,1,-1,-1,-1,-1,1).$$

From Proposition~\ref{67}, we have$A_1=\{\alpha_1\}$ and $A_2=\{\beta_1,\beta_2\}$.

If $ \Delta_\lambda^+\cap A_1\neq\emptyset$, we will have $\IP{\lambda+\rho}{\alpha_1^\vee}>0$. Since $\alpha_1$ is the minimal noncompact root in $\Delta(\fp^+)$, we will have $\IP{\lambda+\rho}{\alpha^\vee}>0$ for all $\alpha \in \Delta(\fp^+)$ and $m(\lambda)=0$. Thus $L(\lambda)$ will be finite-dimensional and $AV(L(\lambda))=\overline{\mathcal{O}_{0}}$. So $k(\lambda)=m(\lambda)=0$.

If $\Delta_\lambda^+\cap A_1=\emptyset$ and $\Delta_\lambda^+\cap A_2\neq\emptyset$, we will have $\IP{\lambda+\rho}{\alpha_1^\vee}\leq 0$ and there is at least one root of $A_2$ 
satisfying $\IP{\lambda+\rho}{\alpha^\vee} > 0$. So we have $m(\lambda)\geq 1$. From the diagram of $\Delta(\fp^+)$ for $\mathfrak{g}_{\mathbb{R}}=\mathfrak{e}_{6(-14)}$, we find that if an antichain $A\subseteq \Delta(\fp^+) $ 
has $m(\lambda)=2$, it must contain $\beta_1$ and $\beta_2$ since they are the smallest  incompatible pair in $\Delta(\fp^+) $.
So we must have $m(\lambda)=1$ since $\Delta_\lambda^+\cap A_2\neq\emptyset$. Thus  $k(\lambda)=m(\lambda)=1$.

If $\Delta_\lambda^+\cap A_2=\emptyset$, $\beta_1$ and $\beta_2$ will satisfy that $\IP{\lambda+\rho}{\alpha^\vee} \leq 0$.  Thus, we have $m(\lambda)\geq 2$. The equality must hold since $m(\lambda)\leq 2$ by the diagram of $\Delta(\fp^+)$. So $k(\lambda)=m(\lambda)=2$.

\subsection{$\mathfrak{g}_{\mathbb{R}}=\mathfrak{e}_{7(-25)}$}

Suppose $L(\lambda)$ is a  highest weight Harish-Chandra module  with highest weight $\lambda$. We denote $\lambda+\rho=(t_1,...,t_8)$. Then from \cite[Lem. 3.17]{EHW:83}, we may suppose $\lambda$ is integral (otherwise $N(\lambda)$ will be irreducible). Also we have $|t_1|< t_2<...<  t_5$, $t_i-t_j\in \mathbb{Z}$ and $2t_i\in \mathbb{Z}$ for all $1\leq i, j\leq 5$.

The Hasse diagram of $\Delta(\fp^+)$ for $\mathfrak{g}_{\mathbb{R}}=\mathfrak{e}_{7(-25)}$ is shown in the Appendix.

The highest root is $\theta= 2\alpha_{1}+2\alpha_{2}+3\alpha_{3}+4\alpha_{4}+3\alpha_{5}+2\alpha_{6}+\alpha_7$.

We denote $$\beta_1=\alpha_7+\alpha_3+\alpha_4+\alpha_5+\alpha_6=(-1,0,0,0,0,1,0,0),$$  $$\beta_2=\alpha_7+\alpha_2+\alpha_4+\alpha_5+\alpha_6=(1,0,0,0,0,1,0,0),$$
$$\gamma_1=\beta_1+\alpha_2+\alpha_4+\alpha_1+\alpha_3=\frac{1}{2}(1,-1,-1,1,-1,1,-1,1),$$
$$\gamma_2=\beta_1+\alpha_1+\alpha_2+\alpha_4+\alpha_5=\frac{1}{2}(1,-1,-1,1,-1,1,-1,1),$$
and 
$$\gamma_3=\beta_1+\alpha_2+\alpha_4+\alpha_5+\alpha_6=(0,0,0,0,1,1,0,0).$$

From Proposition \ref{67}, we have 
$A_1=\{\alpha_7\}$, $A_2=\{\beta_1,\beta_2\}$, and
$A_3=\{\gamma_1,\gamma_2, \gamma_3\}$.

If $ \Delta_\lambda^+\cap A_1\neq\emptyset$, we will have $\IP{\lambda+\rho}{\alpha^\vee}>0$. Since $\alpha_7$ is the minimal noncompact root in $\Delta(\fp^+)$, we will have $\IP{\lambda+\rho}{\alpha^\vee}>0$ for all $\alpha \in \Delta(\fp^+)$ and $m(\lambda)=0$. Thus $L(\lambda)$ will be finite-dimensional and $AV(L(\lambda))=\overline{\mathcal{O}_{0}}$. So $k(\lambda)=m(\lambda)=0$.

If $\Delta_\lambda^+\cap A_1=\emptyset$ and $\Delta_\lambda^+\cap A_2\neq\emptyset$, we will have $\IP{\lambda+\rho}{\alpha_7^\vee} \leq 0$ and there is at least one root of $A_2$ 
satisfying $\IP{\lambda+\rho}{\alpha^\vee}> 0$. So we have $m(\lambda)\geq 1$. From the diagram of $\Delta(\fp^+)$ for $\mathfrak{g}_{\mathbb{R}}=\mathfrak{e}_{7(-25)}$, we find that if an antichain $A\subseteq \Delta(\fp^+) $ 
has $m(\lambda)=2$, it must contain $\beta_1$ and $\beta_2$ since they are the smallest  incompatible pair in $\Delta(\fp^+) $.
So we must have $m(\lambda)=1$ since $\Delta_\lambda^+\cap A_2\neq\emptyset$. Thus  $k(\lambda)=m(\lambda)=1$.

If $\Delta_\lambda^+\cap A_2=\emptyset$ and $\Delta_\lambda^+\cap A_3\neq\emptyset$, $\beta_1$ and $\beta_2$ will satisfy that $\IP{\lambda+\rho}{\alpha^\vee} \leq 0$.  Thus, we have $m(\lambda)\geq 2$.  From the diagram of $\Delta(\fp^+)$ for $\mathfrak{g}_{\mathbb{R}}=\mathfrak{e}_{7(-25)}$, we find that if an antichain $A\subseteq \Delta(\fp^+) $ 
has $m(\lambda)=3$, it must contain $\gamma_1$,  $\gamma_2$ and $\gamma_3$ since they are the smallest  incompatible triples in $\Delta(\fp^+) $.
So we must have $m(\lambda)=2$ since $\Delta_\lambda^+\cap A_3\neq\emptyset$. Thus $k(\lambda)=m(\lambda)=2$.

If $\Delta_\lambda^+\cap A_3=\emptyset$, then $\gamma_1$, $\gamma_2$ and $\gamma_3$ will satisfy that $\IP{\lambda+\rho}{\alpha^\vee} \leq 0$.  Thus, we have $m(\lambda)\geq 3$. The equality must hold since $m(\lambda)\leq 3$ by the diagram of $\Delta(\fp^+)$. So $k(\lambda)=m(\lambda)=3$.

This finishes our second proof of Theorem~\ref{thm:main}. \hfill $\Box$

\newpage
\appendix
\section{Hasse diagrams and distinguished antichains}\label{app:A}

The figures below show the Hasse diagrams of $\gD(\fp^+)$ for $\fg_\R$ in the simply-laced types. Let $c$ be the constant given in \cite[Table on p.~ 115]{EHW:83} and shown below. 
Then the roots in the distinguished antichains  
$A_k=\{\alpha\in \Delta(\fp^+) : \operatorname{ht}(\alpha)=(k-1)c+1\}$, 
$k=1,2,\ldots, r$,  are labelled so that $A_1=\{\alpha_*\}$, $A_2=\{\beta_1,\beta_2\}$, and $A_3=\{\gamma_1,\gamma_2,\gamma_3\}$.

\noindent
$\su(p,q)$, here shown for $p=4$, $q=3$:
\begin{equation*}
\raisebox{-5pc}{
\begin{pspicture}(0,-1)(10,3.5)
\normalsize
\uput[r](7,1){$c=1$}
\cnode*(1.5,0){.07}{a0}
\cnode*(1,.5){.07}{a1}
\cnode*(2,.5){.07}{a2}
\uput[r](1,.5){$\beta_{1}$ }
\uput[r](2,.5){$\beta_{2}$ }
\cnode*(.5,1){.07}{a3}
\cnode*(1.5,1){.07}{a4}
\cnode*(2.5,1){.07}{a5}
\uput[r](.5,1){$\gamma_{1}$ }
\uput[r](1.5,1){$\gamma_{2}$ }
\uput[r](2.5,1){$\gamma_{3}$ }
\cnode*(0,1.5){.07}{a6}
\cnode*(1,1.5){.07}{a7}
\cnode*(2,1.5){.07}{a8}
\cnode*(.5,2){.07}{a9}
\cnode*(1.5,2){.07}{a10}
\cnode*(1,2.5){.07}{a11}
\ncline{->}{a0}{a1}
\ncline{->}{a0}{a2}
\ncline{->}{a1}{a3}
\ncline{->}{a1}{a4}
\ncline{->}{a2}{a4}
\ncline{->}{a3}{a6}
\ncline{->}{a2}{a5}
\ncline{->}{a3}{a7}
\ncline{->}{a4}{a7}
\ncline{->}{a4}{a8}
\ncline{->}{a5}{a8}
\ncline{->}{a6}{a9}
\ncline{->}{a7}{a9}
\ncline{->}{a7}{a10}
\ncline{->}{a8}{a10}
\ncline{->}{a9}{a11}
\ncline{->}{a10}{a11}
\uput[r](1.3,-.3){$\alpha_{4}=\varepsilon_4-\varepsilon_5$}
\uput[r](.8,2.8){$\theta = \varepsilon_1-\varepsilon_7=\alpha_{1}+\alpha_{2}+\alpha_{3}+\alpha_{4}+\alpha_{5}+\alpha_{6}$}
\uput[d](1.1,.4){\scriptsize{3}}
\uput[d](0.6,.9){\scriptsize{2}}
\uput[d](0.1,1.4){\scriptsize{1}}
\uput[u](0.1,1.6){\scriptsize{5}}
\uput[u](0.6,2.1){\scriptsize{6}}
\end{pspicture}
}
\end{equation*}

\noindent
$\so(2,2n-2)$, here shown for $n=6$:
\begin{equation*}
\raisebox{-6.7pc}{
\begin{pspicture}(0,-1)(10,5)
\uput[r](7,2){$c=n-2=4$}
\cnode*(2,0){.07}{a0}
\cnode*(1.5,.5){.07}{a1}
\cnode*(1,1){.07}{a2}
\cnode*(.5,1.5){.07}{a3}
\cnode*(0,2){.07}{a5}
\cnode*(1,2){.07}{a6}
\uput[r](0,2){$\beta_{1}$ }
\uput[r](1,2){$\beta_{2}$ }
\cnode*(.5,2.5){.07}{a7}
\cnode*(1,3){.07}{a9}
\cnode*(1.5,3.5){.07}{a12}
\cnode*(2,4){.07}{b13}
\ncline{->}{a0}{a1}
\ncline{->}{a1}{a2}
\ncline{->}{a2}{a3}
\ncline{->}{a3}{a5}
\ncline{->}{a3}{a6}
\ncline{->}{a5}{a7}
\ncline{->}{a6}{a7}
\ncline{->}{a7}{a9}
\ncline{->}{a9}{a12}
\ncline{->}{a12}{b13}
\uput[d](1.6,.4){\scriptsize{2}}
\uput[d](1.1,.9){\scriptsize{3}}
\uput[d](0.6,1.4){\scriptsize{4}}
\uput[d](0.1,1.9){\scriptsize{5}}
\uput[u](0.1,2.1){\scriptsize{6}}
\uput[u](0.6,2.6){\scriptsize{4}}
\uput[u](1.1,3.1){\scriptsize{3}}
\uput[u](1.6,3.6){\scriptsize{2}}
\uput[r](1.8,-.3){$\alpha_{1}=\varepsilon_1-\varepsilon_2$}
\uput[r](1.8,4.3){$\theta =\varepsilon_1+\varepsilon_2= \alpha_{1}+2\alpha_{2}+2\alpha_{3}+2\alpha_{4}+\alpha_{5}+\alpha_{6}$}
\end{pspicture}
}
\end{equation*}

\noindent
$\so^*(2n)$, here shown for $n=6$:
\begin{equation*}
\raisebox{-2.2pc}{
\begin{pspicture}(0,1)(10,5)
\uput[r](7,2){$c=2$}
\cnode*(2,0){.07}{a0}
\cnode*(1.5,.5){.07}{a1}
\cnode*(1,1){.07}{a2}
\cnode*(2,1){.07}{c2}
\uput[r](1,1){$\beta_{1}$ }
\uput[r](2,1){$\beta_{2}$ }
\cnode*(.5,1.5){.07}{a3}
\cnode*(1.5,1.5){.07}{a4}
\cnode*(0,2){.07}{a5}
\cnode*(1,2){.07}{a6}
\cnode*(2,2){.07}{c6}
\uput[r](0,2){$\gamma_{1}$ }
\uput[r](1,2){$\gamma_{2}$ }
\uput[r](2,2){$\gamma_{3}$ }
\cnode*(.5,2.5){.07}{a7}
\cnode*(1.5,2.5){.07}{a8}
\cnode*(1,3){.07}{a9}
\cnode*(2,3){.07}{a10}
\cnode*(1.5,3.5){.07}{a12}
\cnode*(2,4){.07}{c13}
\ncline{->}{a0}{a1}
\ncline{->}{a1}{a2}
\ncline{->}{a1}{c2}
\ncline{->}{a2}{a3}
\ncline{->}{a2}{a4}
\ncline{->}{c2}{a4}
\ncline{->}{a3}{a5}
\ncline{->}{a3}{a6}
\ncline{->}{a4}{a6}
\ncline{->}{a4}{c6}
\ncline{->}{a5}{a7}
\ncline{->}{a6}{a7}
\ncline{->}{a6}{a8}
\ncline{->}{a7}{a9}
\ncline{->}{c6}{a8}
\ncline{->}{a8}{a9}
\ncline{->}{a8}{a10}
\ncline{->}{a9}{a12}
\ncline{->}{a10}{a12}
\ncline{->}{a12}{c13}
\uput[d](1.6,.4){\scriptsize{4}}
\uput[d](1.1,.9){\scriptsize{3}}
\uput[d](0.6,1.4){\scriptsize{2}}
\uput[d](0.1,1.9){\scriptsize{1}}
\uput[u](0.1,2.1){\scriptsize{5}}
\uput[u](0.6,2.6){\scriptsize{4}}
\uput[u](1.1,3.1){\scriptsize{3}}
\uput[u](1.6,3.6){\scriptsize{2}}
\uput[r](1.8,-.3){$\alpha_{6}=\varepsilon_5+\varepsilon_6$}
\uput[r](1.8,4.3){$\theta = \varepsilon_1+\varepsilon_2=\alpha_{1}+2\alpha_{2}+2\alpha_{3}+2\alpha_{4}+\alpha_{5}+\alpha_{6}$}
\end{pspicture}
}
\end{equation*}

\noindent
$\mathfrak{e}_{6(-14)}$:
\begin{equation*}
\raisebox{-6.5pc}{
	\begin{pspicture}(0,-.5)(10,5.5)	
 \uput[r](7,3){$c=3$}
\psset{linewidth=.5pt,labelsep=8pt,nodesep=0pt}
	\cnode*(2,0){.07}{a0}
	\cnode*(1.5,.5){.07}{a1}
	\cnode*(1,1){.07}{a2}
	\cnode*(.5,1.5){.07}{a3}
	\cnode*(1.5,1.5){.07}{a4}
        \uput[r](0.45,1.5){$\beta_{1}$ }
	\uput[r](1.35,1.5){$\beta_{2}$ }
	\cnode*(0,2){.07}{a5}
	\cnode*(1,2){.07}{a6}
	\cnode*(.5,2.5){.07}{a7}
	\cnode*(1.5,2.5){.07}{a8}
	\cnode*(1,3){.07}{a9}
	\cnode*(2,3){.07}{a10}
	\cnode*(.5,3.5){.07}{a11}
	\cnode*(1.5,3.5){.07}{a12}
	\cnode*(1,4){.07}{a13}
	\cnode*(.5,4.5){.07}{a14}
	\cnode*(0,5){.07}{a15}
	\ncline{->}{a0}{a1}
	\ncline{->}{a1}{a2}
	\ncline{->}{a2}{a3}
	\ncline{->}{a2}{a4}
	\ncline{->}{a3}{a5}
	\ncline{->}{a3}{a6}
	\ncline{->}{a4}{a6}
	\ncline{->}{a5}{a7}
	\ncline{->}{a6}{a7}
	\ncline{->}{a6}{a8}
	\ncline{->}{a7}{a9}
	\ncline{->}{a8}{a9}
	\ncline{->}{a8}{a10}
	\ncline{->}{a9}{a11}
	\ncline{->}{a9}{a12}
	\ncline{->}{a10}{a12}
	\ncline{->}{a11}{a13}
	\ncline{->}{a12}{a13}
	\ncline{->}{a13}{a14}
	\ncline{->}{a14}{a15}
	\uput[r](1.95,0){$\alpha_{1}$ }
	\uput[r](-.25,5.35){$\theta =  \alpha_{1}+2\alpha_{2}+2\alpha_{3}+3\alpha_{4}+2\alpha_{5}+\alpha_{6}$}
	\uput[d](1.6,.5){\scriptsize{3}}
	\uput[d](1.1,1){\scriptsize{4}}
	\uput[d](0.6,1.5){\scriptsize{5}}
	\uput[d](0.1,2){\scriptsize{6}}
	\uput[u](0.1,2){\scriptsize{2}}
	\uput[u](0.6,2.5){\scriptsize{4}}
	\uput[u](1.1,3){\scriptsize{3}}
	\uput[u](1.4,3.5){\scriptsize{5}}
	\uput[u](1.4,3.5){\scriptsize{5}}
	\uput[u](0.9,4.0){\scriptsize{4}}
	\uput[d](0.1,5.0){\scriptsize{2}}
\end{pspicture}
}
\end{equation*}

\vspace{2.5pc}

\noindent
$\mathfrak{e}_{7(-25)}$:
\begin{equation*}
\raisebox{-15pc}{
	\begin{pspicture}(0,-3)(10.5 ,8.5)
 \uput[r](7.2,5.5){$c=4$}
\psset{linewidth=.5pt,labelsep=8pt,nodesep=0pt}
	\cnode*(2.5,-.5){.07}{b0}
	\cnode*(2,0){.07}{a0}
	\cnode*(1.5,.5){.07}{a1}
	\cnode*(1,1){.07}{a2}
	\cnode*(.5,1.5){.07}{a3}
	\cnode*(1.5,1.5){.07}{a4}
	\cnode*(0,2){.07}{a5}
	\cnode*(1,2){.07}{a6}
	\cnode*(.5,2.5){.07}{a7}
	\cnode*(1.5,2.5){.07}{a8}
	\cnode*(1,3){.07}{a9}
	\cnode*(2,3){.07}{a10}
	\cnode*(.5,3.5){.07}{a11}
	\cnode*(1.5,3.5){.07}{a12}
	\cnode*(1,4){.07}{a13}
	\cnode*(.5,4.5){.07}{a14}
	\cnode*(0,5){.07}{a15}
	\cnode*(2.5,3.5){.07}{b10}
	\cnode*(2,4){.07}{b12}
	\cnode*(1.5,4.5){.07}{b13}
	\cnode*(1,5){.07}{b14}
	\cnode*(.5,5.5){.07}{b15}
	\cnode*(1.5,5.5){.07}{c14}
	\cnode*(1,6){.07}{c15}
	\cnode*(1.5,6.5){.07}{c16}
	\cnode*(2,7){.07}{c17}
	\cnode*(2.5,7.5){.07}{c18}
	\ncline{->}{b0}{a0}
	\ncline{->}{a0}{a1}
	\ncline{->}{a1}{a2}
	\ncline{->}{a2}{a3}
	\ncline{->}{a2}{a4}
	\ncline{->}{a3}{a5}
	\ncline{->}{a3}{a6}
	\ncline{->}{a4}{a6}
	\ncline{->}{a5}{a7}
	\ncline{->}{a6}{a7}
	\ncline{->}{a6}{a8}
	\ncline{->}{a7}{a9}
	\ncline{->}{a8}{a9}
	\ncline{->}{a8}{a10}
	\ncline{->}{a9}{a11}
	\ncline{->}{a9}{a12}
	\ncline{->}{a10}{a12}
	\ncline{->}{a11}{a13}
	\ncline{->}{a12}{a13}
	\ncline{->}{a13}{a14}	
	\ncline{->}{a14}{a15}
	\ncline{->}{b10}{b15}
	\ncline{->}{a10}{b10}
	\ncline{->}{a12}{b12}
	\ncline{->}{a13}{b13}
	\ncline{->}{a14}{b14}
        \ncline{->}{b13}{b14}
        \ncline{->}{b12}{b13}
        \ncline{->}{b10}{b12}
	\ncline{->}{a15}{b15}
	\ncline{->}{b14}{c14}
	\ncline{->}{c15}{b15}
	\ncline{->}{c14}{c15}
	\ncline{->}{c15}{c16}
        \ncline{->}{c16}{c17}
        \ncline{->}{c17}{c18}
	\uput[d](2.1,-.1){\scriptsize{6}}
	\uput[d](1.6,.4){\scriptsize{5}}
	\uput[d](1.1,.9){\scriptsize{4}}
	\uput[d](0.6,1.4){\scriptsize{3}}
	\uput[d](0.1,1.9){\scriptsize{1}}
	\uput[u](0.1,2.1){\scriptsize{2}}
	\uput[u](0.6,2.6){\scriptsize{4}}
	\uput[u](0.6,3.6){\scriptsize{5}}
	\uput[u](1.9,4.1){\scriptsize{3}}
	\uput[u](1.4,4.6){\scriptsize{4}}
	\uput[u](1.4,5.6){\scriptsize{2}}
	\uput[u](0.1,5.1){\scriptsize{6}}
	\uput[u](0.6,5.6){\scriptsize{5}}
	\uput[u](1.1,6.1){\scriptsize{4}}
	\uput[u](1.6,6.6){\scriptsize{3}}
	\uput[u](2.1,7.1){\scriptsize{1}}
	\uput[d](2.5,-.5){$\alpha_{7}$}
	\uput[r](0.4,1.5){$\beta_{1}$}
	\uput[r](1.4,1.5){$\beta_{2}$}
	\uput[r](0.4,3.5){$\gamma_{1}$}
	\uput[r](1.4,3.5){$\gamma_{2}$}
	\uput[r](2.4,3.5){$\gamma_{3}$}
	\uput[r](2.3,7.8){$\theta=2\alpha_1+2\alpha_2+3\alpha_3+4\alpha_4+2\alpha_6+\alpha_7$}
\end{pspicture}
}
\end{equation*}

\subsection*{Acknowledgments}
	  Z. Bai was supported  by the National Natural Science Foundation of
	China (No. 12171344) and the National Key $\textrm{R}\,\&\,\textrm{D}$ Program of China (No. 2018YFA0701700 and No. 2018YFA0701701). 	X. Xie was supported  by the National Natural Science Foundation of
	China (No. 12171030 and No.  11801031).

\newpage

\end{document}